\documentclass[final,hidelinks,onefignum,onetabnum]{siamart220329}

\usepackage{amscd,amssymb, amsmath,amsfonts, wasysym, mathrsfs, enumerate, mathtools,hhline,xcolor}%
\usepackage{graphicx}
\usepackage[all, cmtip]{xy}
\usepackage{multirow}

\definecolor{hot}{RGB}{65,105,225}
\usepackage{algorithm}
\usepackage[noend]{algpseudocode}

\usepackage{hyperref}
\hypersetup{
	plainpages=false,
    colorlinks,
    linktocpage=true,
    linkcolor=hot,
    citecolor=hot,
    urlcolor=hot
}
\usepackage[hyperpageref]{backref}

\usepackage[colorinlistoftodos,prependcaption,textsize=footnotesize]{todonotes}

\makeatletter
\@mparswitchfalse%
\makeatother
\normalmarginpar %

\newtheorem{remark}[theorem]{Remark}
\newtheorem{assumption}[theorem]{Assumption}

\begin{document}

\title{A parallel domain decomposition method for solving elliptic equations on manifolds}

\author{Lizhen Qin \thanks{School of Mathematics, Nanjing University, Nanjing, Jiangsu, China
  (\email{qinlz@nju.edu.cn}).}
\and Feng Wang \thanks{Key Laboratory of Numerical Simulation of Large Scale Complex Systems (Nanjing Normal University), Ministry of Education; Jiangsu International Joint Laboratory of Big Data Modeling, Computing, and Application; School of Mathematical Sciences, Nanjing Normal University, Nanjing Jiangsu, China (\email{fwang@njnu.edu.cn}).}
\and Yun Wang \thanks{School of Mathematical Sciences, Center for Dynamical Systems and Differential Equations, Soochow University, Suzhou, China (\email{ywang3@suda.edu.cn}).}}

\maketitle

\headers{DDM on Manifolds}{L. Qin, F. Wang, and Y. Wang}

\begin{abstract}
We propose a new numerical domain decomposition method for solving elliptic equations on compact Riemannian manifolds. One advantage of this method is its ability to bypass the need for global triangulations or grids on the manifolds. Additionally, it features a highly parallel iterative scheme. To verify its efficacy, we conduct numerical experiments on some $4$-dimensional manifolds without and with boundary.
\end{abstract}

\begin{AMS}{Primary 65N30; Secondary 58J05, 65N55.}
\end{AMS}


\section{Introduction}\label{sec_introduction}
Elliptic equations on Riemannian manifolds are important both in analysis and in geometry (see e.g. \cite{schoen_yau} and \cite{jost}).
These equations appear in many areas, such as multifluid dynamics, micromagnetics, and image processing (see e.g. \cite{Bansch_Morin_Nochetto,mohamed2005finite,reuter2009discrete,dobrev2010surface,Holst_Stern,Holst2016Wave,bonito2020divergence,bachini2021intrinsic,jankuhn2021error,jin2021gradient,Barrera_Kolev_2023}). Recently, numerical methods for high dimensional spheres are developed using machine learning techniques (see e.g. \cite{LLSZZ}). A simple and important example of such equations is
\begin{equation}
\label{eqn_laplace_problem}
- \Delta u + bu = f.
\end{equation}
Here $\Delta$ is the Laplace-Beltrami operator, or Laplacian for brevity, defined on a $d$-dimensional Riemannian manifold $M$. Many manifolds are naturally submanifolds of Euclidean spaces, where the ``dimension'' of a submanifold is referred to as the topological dimension of the manifold, not the dimension of its ambient Euclidean space. For example, the $n$-dimensional unit sphere $S^n$ is embedded in $\mathbb{R}^{n+1}$.

When the manifold $M$ is a two dimensional Riemannian submanifold in $\mathbb{R}^{3}$, i.e. a surface, the numerical methods to solve differential equations, particularly (\ref{eqn_laplace_problem}), on $M$ have been extensively studied for a long history (see e.g. \cite{nedelec_planchard}, \cite{nedelec}, \cite{baumgardner_frederickson}, \cite{dziuk88} and \cite{dziuk91}). A conventional and popular approach is to solve the equations by finite element methods based on a global grid or triangulation on $M$. Such a grid can be obtained by polyhedron approximation or level set method. This approach has been far-reaching developed and widely applied (see e.g. \cite{demlow2007adaptive,dziuk2007finite,demlow2009higher,antonietti2015high,reusken2015analysis,olshanskii2017trace}, see also e.g. \cite{DDE05,dziuk_elliott,BDN20} for surveys and bibliographies).

However, practically, it's of essential difficulty to build a global triangulation on a higher dimensional manifold for the purpose of numerical computation. The difficulty comes from at least three aspects. Firstly, high dimensional spaces challenge significantly people's geometric intuition. The experience of tessellation in $\mathbb{R}^{2}$ and $\mathbb{R}^{3}$ cannot be simply ported to higher dimensions. Secondly, many important and interesting examples of high dimensional manifolds are not submanifolds of Euclidean spaces by definition. A case in point is the complex projective spaces $\mathbb{CP}^{n}$ which is foundational to algebraic geometry. The usual polytopal approximation or level set method is not applicable to such manifolds. Thirdly, even if $M$ is a submanifold of some $\mathbb{R}^{k}$ by definition, if the codimension of $M$ in $\mathbb{R}^{k}$ is greater than $1$, which is often the case, the triangulation of $M$ is horribly difficult in general due to topological and geometrical complexity. We do acknowledge that J.~H.~C.~Whitehead proved (\cite{whitehead}) that every smooth manifold can be globally triangulated in an abstract way. However, to the best of our knowledge, in practice, there has been no efficient algorithm to construct concrete triangulations over a high dimensional manifold in general.

To circumvent the above difficulty, in \cite{qin_zhang_zhang}, Qin--Zhang--Zhang proposed a new idea to numerically solve elliptic PDEs on manifolds by avoiding global triangulations. Since a $d$-dimensional manifold $M$ has local coordinate charts by definition, $M$ can be decomposed into finitely many subdomains that carry local coordinates. Consequently, an elliptic equation on each subdomain can be transformed to one on a domain in $\mathbb{R}^{d}$. Thus an elliptic problem on $M$ can be solved by Domain Decomposition Methods (DDMs) with subproblems on Euclidean domains.

The idea of \cite{qin_zhang_zhang} was developed by Cao--Qin in \cite{cao_qin}. They extended a DDM proposed by P.~L.~Lions (\cite[Section~I.~4]{lions1}) for solving problems on Euclidean domains to a DDM on manifolds. In this paper, we combine the ideas of \cite{qin_zhang_zhang} and \cite{lions2}, which yields an adaption and generalization of Lions' another DDM (\cite[p.~66]{lions2}). This DDM was also proposed by Lions for solving continuous problems on Euclidean domains. We found that it can also be well-adapted to a DDM on manifolds. One notable advantage of this method is its high level of parallelism, making it particularly valuable when dealing with a large number of subdomains.

Here is an overview of the method in this paper. First of all, global triangulation/grid is completely avoided. There is also no requirement of any compatibility among local grids, i.e. grids over subdomains. (As explained in the introduction of \cite{cao_qin}, this incompatibility actually reflects the high flexibility of the numerical methods.) The numerical difficulty caused by the nonmatching grids is resolved by the technique of interpolation. Furthermore, the subproblems on subdomains can be localized and then be reduced to problems on Euclidean domains. These features bear resemblance to that method described in \cite{cao_qin}. However, this paper is remarkably different from \cite{cao_qin} in that the current method is highly parallel. In fact, this work follows \cite{lions2}, while \cite{cao_qin} follows \cite{lions1}. Lions described the work \cite[Section~I.~4]{lions1} as ``sequential" (see the sentence right before (26) in \cite{lions1}), which means the subproblems on subdomains have to be solved one by one sequentially. On the contrary, the work \cite{lions2} was claimed as ``parallel" (see a few sentences before (37) in \cite{lions2}), which means the subproblems on subdomains can be solved independently and parallelly at each iterative step.

It is worth noting that, even in the special case where the target manifold is a Euclidean domain, our algorithm (see Algorithms \ref{alg_continuous} and \ref{alg_discrete}) is different from \cite[p.~66]{lions2} though we follow it largely. Typically, we employ a technique of partition of unity which was not utilized in \cite[p.~66]{lions2}. Due to this technique, we obtain a better convergence theory (see Theorem \ref{thm_convergence} and Remark \ref{rmk_convergence}).

Overall, our method is an overlapping DDM on manifolds which avoids global grids. There has been a long history of DDM. It was originated from the work of H.~A.~Schwarz \cite{schwarz} who invented his famous alternating method. The seminal works \cite{lions1,lions2} of P.~L.~Lions are natural and wonderful extensions of Schwarz alternating method. From then on, the area of DDM has been developed explosively and greatly (see e.g. \cite{quarteroni_valli,smith_bjorstad_gropp,toselli_widlund,dolean2015DDM}). The later development often take DDM schemes, including both overlapping and nonoverlapping ones, as preconditioners for globally discretized problems (see e.g. \cite{BPWX,xu92,1996CaiSaadOverlapping,xu_zou,dryja2013}). However, a globally discretized problem on a manifold $M$ should be based on a global grid on $M$, which is not accessible in our case. Therefore, we follow more closely Lions' original approach rather than the later development.

Our DDM in this paper is numerically verified on examples of $4$-dimensional manifolds.
They are $\mathbb{CP}^{2}$, $B^{4}$ and $B^{2} \times S^{2}$, where $\mathbb{CP}^{2}$ is the complex projective plane, $S^{2}$ is the $2$-dimensional unit sphere, and $B^{d}$ is the $d$-dimensional unit ball. The $\mathbb{CP}^{2}$ is without boundary, while the latter two are with boundary. The numerical results show that our method solves effectively the equation \eqref{eqn_laplace_problem} on these manifolds.

The outline of this paper is as follows. In Section \ref{sec_continuous}, we shall adapt and extend P.~L.~Lions' method in \cite[p.~66]{lions2} to a DDM for continuous problems on manifolds and prove its convergence. In Section \ref{sec_scheme}, we shall propose our numerical DDM which is a discrete imitation of Lions' method. In Section \ref{sec_boundary}, we shall introduce a technique to decompose manifolds with boundary. Finally, some numerical results will be presented in Section \ref{sec_experiment}.

\section{Theory on Continuous Problems}\label{sec_continuous}
In this section, we shall firstly formulate an elliptic model problem of second order on manifolds. Then we introduce a domain decomposition method (DDM) to solve the model problem (see Algorithm \ref{alg_continuous} below). This DDM iteration procedure was originally proposed by P.~L.~Lions in \cite[p.~66]{lions2} for solving differential equations in a Euclidean domain. Unlike those DDMs in \cite{lions1} and \cite{cao_qin}, this iterative procedure is highly parallel. More precisely, the subproblem in each subdomain can be solved independently at each iterative step. This procedure motivates us to propose a highly parallel numerical procedure (Algorithm \ref{alg_discrete}) in the next section which provides numerical approximations to the solution to the model problem.

\subsection{Model Problem and Continuous Algorithm}
We shall formulate a model problem, the continuous Algorithm \ref{alg_continuous}, and Theorem \ref{thm_convergence} on convergence. The proof of convergence is much more complicated than those theorems in \cite{lions1} and \cite{cao_qin}, and will occupy the next two subsections.

Let $M$ be a $d$-dimensional compact smooth manifold without or with boundary $\partial M$. Equip $M$ with a Riemannian metric $g$. Then the Laplace operator $\Delta$, also named the Laplace-Beltrami operator, is defined on $M$. Note that neither $g$ nor $\Delta$ can be expressed by coordinates globally in general because $M$ does not necessarily have a global coordinate chart. In a local chart with coordinates $(x_{1}, \dots, x_{d})$, the Riemannian metric tensor $g$ is expressed as
\begin{equation}\label{eqn_metric}
g= \sum_{\alpha, \beta=1}^{d} g_{\alpha \beta} \mathrm{d} x_{\alpha} \otimes \mathrm{d} x_{\beta},
\end{equation}
where the matrix $(g_{\alpha \beta})_{d \times d}$ is symmetric and positive definite. The Laplace operator $\Delta$ is expressed as
\begin{align}\label{eqn_Laplace}
    \Delta u & = \frac{1}{\sqrt{G}} \sum_{\alpha=1}^{d} \frac{\partial}{\partial x_{\alpha}} \left( \sum_{\beta=1}^{d} g^{\alpha \beta} \sqrt{G} \frac{\partial u}{\partial x_{\beta}} \right) \nonumber \\
    & = \sum_{\alpha,\beta=1}^{d} g^{\alpha \beta}  \frac{\partial^{2} u}{\partial x_{\alpha} \partial x_{\beta}} + \frac{1}{\sqrt{G}} \sum_{\beta=1}^{d} \sum_{\alpha=1}^{d} \frac{\partial}{\partial x_{\alpha}} \left( g^{\alpha \beta} \sqrt{G} \right) \frac{\partial u}{\partial x_{\beta}},
\end{align}
where $G = \det \left( (g_{\alpha \beta})_{d \times d} \right)$ is the determinant of the matrix $(g_{\alpha \beta})_{d \times d}$ and $(g^{\alpha \beta})_{d \times d}$ is the inverse of $(g_{\alpha \beta})_{d \times d}$. It's well-known that $\Delta$ is an elliptic differential operator of second order.

We consider the following model problem on $M$
\begin{equation}\label{eqn_problem}
\left\{
\begin{array}{rcl}
Lu= - \Delta u + bu & = & f, \\
u|_{\partial M} & = & \varphi,
\end{array}
\right.
\end{equation}
where $b \geq 0$ is a constant and $f \in L^{2} (M)$. It suffices to solve \eqref{eqn_problem} on each component of $M$. Therefore, without loss of generality, $M$ is assumed to be connected. Furthermore, it is possible for $\partial M$ to be either empty or nonempty. In the case where $\partial M \neq \emptyset$, we assume $\varphi \in C^{0} (\partial M) \cap H^{\frac{1}{2}} (\partial M)$. On the other hand, if $\partial M = \emptyset$, the boundary condition $u|_{\partial M} = \varphi$ is vacuously satisfied. However, to guarantee \eqref{eqn_problem} is well-posed, we need to further assume $b>0$ if $\partial M = \emptyset$. With these assumptions, the problem \eqref{eqn_problem} always has a unique solution $u \in C^{0} (M) \cap H^{1} (M)$. By \eqref{eqn_Laplace}, the local problems of \eqref{eqn_problem} are general elliptic equations with varying coefficients.

Now we describe a domain decomposition iteration procedure to solve (\ref{eqn_problem}). This method was originally proposed by P.~L.~Lions in \cite[p.~66]{lions2} for solving differential equations in a Euclidean domain. We found that the method is well-adapted to solve problems on manifolds. Suppose $M$ is decomposed into $m$ subdomains, i.e. $M = \bigcup_{i=1}^{m} M_{i}$. Here $M_{i}$ is a closed subdomain (submanifold with codimension $0$) of $M$ with Lipschitz boundary. Let $\gamma_{i} = \partial M_{i} \setminus \partial M$. We make the following assumption on decomposition.

\begin{assumption}\label{asp_decomposition}
Suppose $M = \bigcup_{i=1}^{m} (M_{i} \setminus \overline{\gamma_{i}})$, and $S: = \bigcup_{i \neq j} (\gamma_{i} \cap \gamma_{j})$ has a finite $(d-2)$-dimensional Hausdorff measure.
\end{assumption}

The requirement of the Hausdorff measure of $S$ can be easily satisfied. For example, we may choose $M_{i}$ with piecewise $C^{1}$ boundary and arrange that all $\gamma_{i}$ meet transversely with each other. Transversality is a generic property which, therefore, can be satisfied almost trivially. (For a detailed theory on transversality, see e.g. \cite[Chapter~3]{hirsch}.) However, a nonoverlapping decomposition, commonly used in DDMs, does not satisfy Assumption \ref{asp_decomposition}. Our DDM is an overlapping one.

Suppose $\{ \rho_{i} \mid 1 \leq i \leq m \}$ is a partition of unity of $M$, i.e. these $\rho_{i}$ are a family of nonnegative functions on $M$ and $\sum_{i=1}^{m} \rho_{i} = 1$.

\begin{assumption}\label{asp_partition}
The partition of unity  $\{ \rho_{i} \mid 1 \leq i \leq m \}$ is subordinate to the decomposition in Assumption \ref{asp_decomposition}, i.e. $\mathrm{supp} \rho_{i} \subset M_{i} \setminus \overline{\gamma_{i}}$, where $\mathrm{supp} \rho_{i}$ is the support of $\rho_{i}$. Furthermore, each $\rho_{i}$ is Lipschitz continuous.
\end{assumption}

Now we are ready to formulate the desired algorithm of solving continuous problem \eqref{eqn_problem}.

\begin{algorithm}
\caption{A DDM to solve \eqref{eqn_problem}.}
\label{alg_continuous}

\begin{algorithmic}[1]
\State%
Choose an arbitrary initial guess $u^{0} \in C^{0} (M) \cap H^{1} (M)$ with $u^{0}|_{\partial M} = \varphi$ for \eqref{eqn_problem}.

\State%
For each $n>0$, assuming $u^{n-1}$ has been obtained, for $1 \leq i \leq m$, find a function $u^{n}_{i}$ on $M_{i}$ such that
\begin{equation}\label{alg_continuous_1}
\left\{
\begin{aligned}
L u^{n}_{i} (x) & = f(x), & x \in M_{i} \setminus \partial M_{i};
\\
u^{n}_{i} (x) & = u^{n-1} (x), & x \in \partial M_{i}.
\end{aligned}
\right.
\end{equation}

\State%
Let $u^{n} = \sum_{i=1}^{m} \rho_{i} u^{n}_{i}$.
\end{algorithmic}
\end{algorithm}

Unlike those methods in \cite[Section~I.~4]{lions1} and \cite{cao_qin}, Algorithm \ref{alg_continuous} is highly parallel. Obviously, the subproblems in \eqref{alg_continuous_1} are solved independently on $M_{i}$ for all $i$. It's necessary to point out that $u^{n}_{i}$ in \eqref{alg_continuous_1} is merely defined on $M_{i}$. However, $u^{n}$ is well-defined on the whole $M$ since $\mathrm{supp} \rho_{i} \subset M_{i}$. (More precisely, the function $\rho_{i} u^{n}_{i}$ on $M_{i}$ is extended as $0$ outside of $M_{i}$.) We shall prove that $\{ u^{n}_{i} \}$ and $\{ u^{n} \}$ converge to the solution $u$ in certain sense.

The boundary value condition of $u^{n}_{i}$ in Algorithm \ref{alg_continuous} is slightly different from Lions' original approach in \cite{lions2}. More precisely, a partition of unity was not employed in \cite{lions2}, instead, the boundary value $u^{n}_{i}|_{\gamma_{i}}$ was chosen as an arbitrary function $v$ such that (see \cite[(42)~\&~(43)]{lions2})
\[
\min \{ u^{n-1}_{j} (x) \mid x \in M_{j}, j \neq i \} \leq v(x) \leq \max \{ u^{n-1}_{j} (x) \mid x \in M_{j}, j \neq i \}.
\]
Our $u^{n}_{i}|_{\gamma_{i}}$ in Algorithm \ref{alg_continuous} certainly satisfies the above inequalities. Moreover, by a partition of unity satisfying Assumption \ref{asp_partition}, the iteration in Algorithm \ref{alg_continuous} even has a better theory of convergence, see Theorem \ref{thm_convergence} and Remark \ref{rmk_convergence} below.

First of all, the $u^{n}_{i}$ in \cite{lions2} was usually not continuous on $M_{i}$ since $u^{n}_{i}|_{\gamma_{i}}$ was not necessarily continuous. However, the following lemma shows our $u^{n}_{i}$ in \eqref{alg_continuous_1} has better regularity. Let $u^{0}_{i} = u^{0}|_{M_{i}}$. In the following, all statements on $u^{n}_{i}|_{\partial M}$ are vacuously true if $M_{i} \cap \partial M = \emptyset$.
\begin{lemma}
Under the Assumptions \ref{asp_decomposition} and \ref{asp_partition}, in \eqref{alg_continuous_1}, we have $\forall n$, $\forall i$, $u^{n} \in C^{0} (M) \cap H^{1} (M)$, $u^{n}_{i} \in C^{0} (M_{i}) \cap H^{1} (M_{i})$, $u^{n}|_{\partial M} = \varphi$, and $u^{n}_{i}|_{\partial M} = \varphi$.
\end{lemma}
\begin{proof}
We already know the conclusion holds for $n=0$. Suppose the conclusion is true for $n-1$. Then $u^{n-1} - u \in C^{0} (M) \cap H^{1} (M)$. Furthermore, $u^{n}_{i}$ satisfies
\[
\left\{
\begin{aligned}
L (u^{n}_{i} -u) (x) & = 0, & x \in M_{i} \setminus \partial M_{i};
\\
(u^{n}_{i} -u) (x) & = (u^{n-1} -u) (x), & x \in \partial M_{i}.
\end{aligned}
\right.
\]
By \cite[Corollary~8.28]{Gilbarg_Trudinger}, we infer $u^{n}_{i} - u \in C^{0} (M_{i}) \cap H^{1} (M_{i})$. As $u \in C^{0} (M) \cap H^{1} (M)$, we see $u^{n}_{i} \in C^{0} (M_{i}) \cap H^{1} (M_{i})$. By Rademacher's theorem (see Theorem 5 in \cite[\S4.2.3]{evans_gariepy}), the Lipschitz continuity of $\rho_{i}$ means $\rho_{i} \in W^{1,\infty} (M)$, which implies $u^{n} \in C^{0} (M) \cap H^{1} (M)$. Since $u^{n-1}|_{\partial M} = \varphi$, we have $u^{n}_{i}|_{\partial M} = \varphi$ and hence $u^{n}|_{\partial M} = \varphi$.
\end{proof}

Secondly, we have the following theorem on convergence.
\begin{theorem}\label{thm_convergence}
Under the Assumptions \ref{asp_decomposition} and \ref{asp_partition}, suppose the solution $u \in C^{0} (M) \cap H^{1} (M)$. Assume further $u$ and the initial guess $u^{0}$ are Lipschitz continuous if $\partial M \neq \emptyset$. Let
\[
\theta^{n} = \max\{ \| u^{n}_{i} - u \|_{C^{0}(M_{i})} \mid 1 \leq i \leq m \}.
\]
Then the following holds:
\begin{enumerate}[(1)]
\item $\lim\limits_{n \rightarrow \infty} \theta^{n} =0$;

\item $\forall n>0$, $\| u^{n} - u \|_{C^{0}(M)} \leq \theta^{n}$ and $\| u^{n} - u \|_{H^{1}(M)} \leq C_{1} \theta^{n}$;

\item $\forall n>1$, $\max\{ \| u^{n}_{i} - u \|_{H^{1}(M_{i})} \mid 1 \leq i \leq m \} \leq C_{2} \theta^{n-1}$.
\end{enumerate}
Here $C_{1}$ and $C_{2}$ are constants independent of $n$ and $i$.
\end{theorem}

Theorem \ref{thm_convergence} was essentially proved by Lions in \cite[p.~66-67]{lions2} in the case that $M$ is a Euclidean domain. Even in that special case, the proof was complicated. Actually, the proof in \cite{lions2} was presented in a dense style. We shall follow and adapt Lions' argument and then extend it to the case of general manifolds. Our proof will be given in the next two subsections. To ensure the readability for the reader, we shall provide plenty of details of argument.

\begin{remark}\label{rmk_convergence}
Even in the case that $M$ is a Euclidean domain, the statement on convergence in \cite{lions2} is slightly different from Theorem \ref{thm_convergence}. In \cite{lions2}, the function $u^{n}_{i}$ was not necessarily continuous on $M_{i}$, the uniform convergence and $H^{1}$-convergence of $u^{n}_{i}$ on $M_{i}$ were not claimed either. It was merely proved that $u^{n}_{i}$ converges uniformly to $u$ on each compact set $K_{i} \subset M_{i} \setminus \overline{\gamma_{i}}$ there. On the other hand, our function $u^{n}_{i}$ is continuous on $M_{i}$, we also obtain the uniform convergence and $H^{1}$-convergence of $u^{n}_{i}$ on $M_{i}$. This is mainly because we employ the technique of partition of unity. The benefits of this technique can be seen in the proof of Theorem \ref{thm_convergence}.
\end{remark}

\subsection{Subsolutions and Supersolutions}
The key ingredient of the proof of Theorem \ref{thm_convergence} is the technique of subsolutions and supersolutions. By considering $u^{n}_{i} - u$, it suffices to study the following homogeneous problem on $M$
\begin{equation}\label{eqn_homogenous}
\left\{
\begin{array}{rcl}
L \hat{u} & = & 0, \\
\hat{u}|_{\partial M} & = & 0,
\end{array}
\right.
\end{equation}
which obviously has a unique solution $\hat{u} =0$.

As mentioned before, $L$ can be expressed as
\begin{equation}\label{eqn_operator_local}
Lu = - \frac{1}{\sqrt{G}}\sum_{\alpha=1}^{d}\frac{\partial}{\partial x_{\alpha}} \left( \sum_{\beta=1}^{d}g^{\alpha \beta} \sqrt{G} \frac{\partial u}{\partial x_{\beta}} \right) + bu
\end{equation}
in a local chart with coordinates $(x_{1}, \dots, x_{d})$. Let's recall the definition of subsolutions and supsolutions in the sense of viscosity (cf. \cite[Definition~2.2]{crandall_ishii_lions}).

\begin{definition}\label{def_sub_super}
A function $\underline{u}$ (resp. $\bar{u}$) on $M \setminus \partial M$ is a subsolution (resp. supersolution) of the equation $L \hat{u} =0$ if it satisfies:
\begin{enumerate}[(1)]
\item $\underline{u}$ (resp. $\bar{u}$) is upper (resp. lower) semicontinuous;

\item for each $\hat{x} \in M \setminus \partial M$ and each $C^{2}$ function $\psi$ on $M \setminus \partial M$ such that $\hat{x}$ is a local maximum (resp. minimum) of $\underline{u} - \psi$ (resp. $\bar{u} - \psi$), we have
    \begin{align*}
    - \frac{1}{\sqrt{G}}\sum_{\alpha=1}^{d}\frac{\partial}{\partial x_{\alpha}} \left( \sum_{\beta=1}^{d}g^{\alpha \beta} \sqrt{G} \frac{\partial \psi}{\partial x_{\beta}} \right) (\hat{x}) + b \underline{u} (\hat{x}) & \leq 0 \\
    (\text{resp.}~~ - \frac{1}{\sqrt{G}}\sum_{\alpha=1}^{d}\frac{\partial}{\partial x_{\alpha}} \left( \sum_{\beta=1}^{d}g^{\alpha \beta} \sqrt{G} \frac{\partial \psi}{\partial x_{\beta}} \right) (\hat{x}) + b \bar{u} (\hat{x}) & \geq 0 ),
    \end{align*}
    where $L$ is expressed as \eqref{eqn_operator_local} in a local chart containing $\hat{x}$.
\end{enumerate}
\end{definition}

\begin{remark}
There is also a notion of subsolutions and supersolutions in the sense of distribution. It is actually equivalent to that in the sense of viscosity (see e.g. \cite[Appendix]{lions1983}, \cite{ishii}, and \cite{harvey_lawson2013}).
\end{remark}

Consider two domain decomposition iterative schemes as follows: Choose arbitrarily an initial guess $\underline{u}^{0} \in C^{0} (M)$ (resp. $\bar{u}^{0} \in C^{0} (M)$) with $\underline{u}^{0}|_{\partial M} = \bar{u}^{0}|_{\partial M} =0$ such that $\underline{u}^{0}$ (resp. $\bar{u}^{0}$) is a subsolution (resp. supersolution) of \eqref{eqn_homogenous}. Let $\underline{u}^{0}_{i} = \underline{u}^{0}|_{M_{i}}$. For each $n \geq 1$, define
\begin{equation}\label{eqn_DDM_sub}
\left\{
\begin{aligned}
L \underline{u}^{n}_{i} (x) & = 0, & x \in M_{i} \setminus \partial M_{i};
\\
\underline{u}^{n}_{i} (x) & = \min \{ \underline{u}^{n-1}_{j} (x) \mid x \in M_{j}, j \neq i \}, & x \in \gamma_{i};
\\
\underline{u}^{n}_{i} (x) & = 0, & x \in \partial M_{i} \cap \partial M.
\end{aligned}
\right.
\end{equation}
Similarly, let $\bar{u}^{0}_{i} = \bar{u}^{0}|_{M_{i}}$. For each $n \geq 1$, define
\begin{equation}\label{eqn_DDM_sup}
\left\{
\begin{aligned}
L \bar{u}^{n}_{i} (x) & = 0, & x \in M_{i} \setminus \partial M_{i};
\\
\bar{u}^{n}_{i} (x) & = \max \{ \bar{u}^{n-1}_{j} (x) \mid x \in M_{j}, j \neq i \}, & x \in \gamma_{i};
\\
\bar{u}^{n}_{i} (x) & = 0, & x \in \partial M_{i} \cap \partial M.
\end{aligned}
\right.
\end{equation}
Here $\underline{u}^{n}_{i}|_{\partial M_{i}}$ and $\bar{u}^{n}_{i}|_{\partial M_{i}}$ are not necessarily continuous. We take $\underline{u}^{n}_{i}$ and $\bar{u}^{n}_{i}$ as the Perron solutions with Dirichlet boundary value on $M_{i}$ (see \cite[Theorem~6.11]{Gilbarg_Trudinger}). The goal of this subsection is to prove Proposition \ref{prop_sub_super} below.

In the following, all arguments cover both the cases $\partial M \neq \emptyset$ and $\partial M = \emptyset$. In the case of $\partial M = \emptyset$, the proof is even simpler. All statements related to $\partial M$ and $M_{i} \cap \partial M$ are vacuously true if $\partial M = \emptyset$ and $M_{i} \cap \partial M = \emptyset$ respectively.
\begin{lemma}\label{lem_sequence_continuous}
$\{ \underline{u}^{n}_{i} \}$ (resp. $\{ \bar{u}^{n}_{i} \}$) is an increasing (resp. decreasing) sequence of solutions to $L \hat{u} =0$ with an upper (resp. lower) bound $0$. Furthermore, $\{ \underline{u}^{n}_{i} \}$ and $\{ \bar{u}^{n}_{i} \}$ are continuous on $M_{i} \setminus (\gamma_{i} \cap S)$. Here $S$ is the set defined in Assumption \ref{asp_decomposition}.
\end{lemma}
\begin{proof}
We only prove the case of $\{ \underline{u}^{n}_{i} \}$ since the proof of the other case is similar. The argument is an induction on $n$. The conclusion is true for $n=0$ by assumption.

Firstly, since $L\underline{u}^{1}_{i}  =0$, $\underline{u}^{1}_{i}|_{\partial M_{i}} = \underline{u}^{0}$ is continuous, and $\partial M_{i}$ is Lipschitz, we have $\underline{u}^{1}_{i} \in C^{0} (M_{i})$. By the assumption that $\underline{u}^{0}$ is a subsolution and $\underline{u}^{0}|_{\partial M} = 0$, we infer $\underline{u}^{0} \leq 0$, which further implies $\underline{u}^{0}_{i} = \underline{u}^{0}|_{M_{i}} \leq 0$. Since $\underline{u}^{1}_{i}$ is a solution and $\underline{u}^{1}_{i}|_{\partial M_{i}} = \underline{u}^{0}_{i}|_{\partial M_{i}} \leq 0$, we have $\underline{u}^{0}_{i} \leq \underline{u}^{1}_{i} \leq 0$. The conclusion is verified for $n=1$.

Suppose the conclusion is true for $n \leq k-1$. By inductive hypothesis, $\forall i$, $\underline{u}^{k-2}_{i} \leq \underline{u}^{k-1}_{i} \leq 0$ and $\underline{u}^{k-1}_{i}$ is continuous in $M_{i} \setminus (\gamma_{i} \cap S)$.
\begin{figure}[htbp]
\centering
  \includegraphics[width=0.3\textwidth]{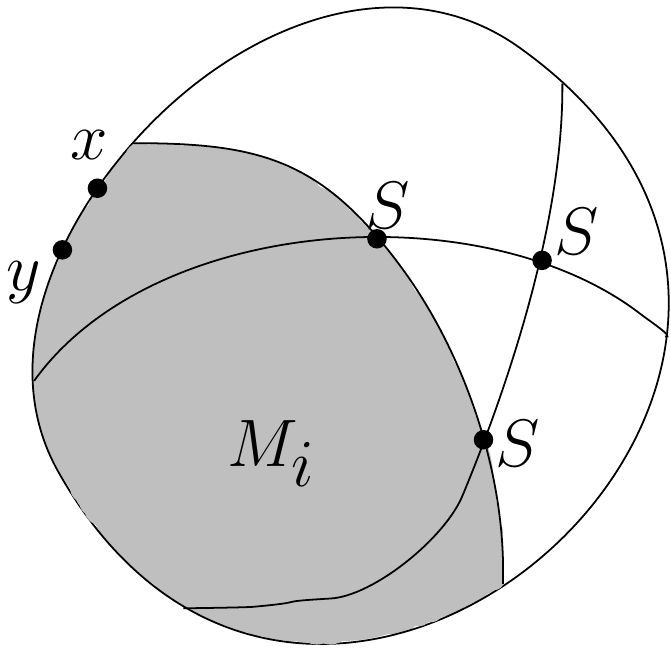} \qquad  \includegraphics[width=0.3\textwidth]{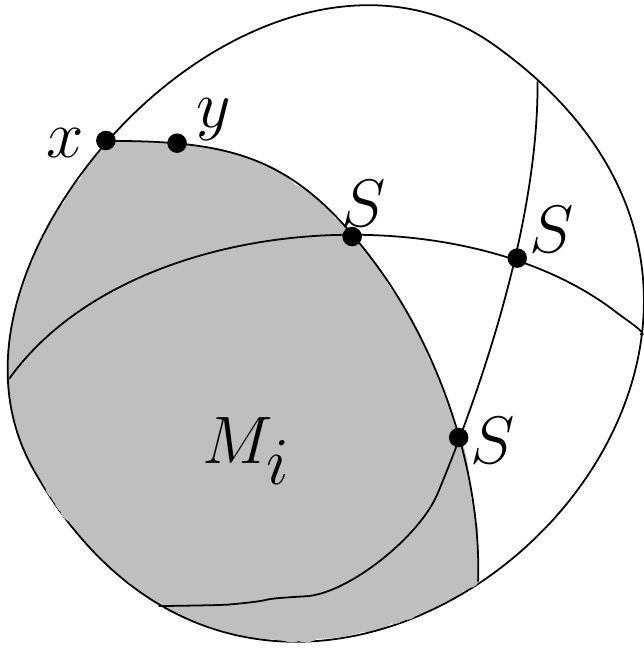}
  \caption{An illustration of domain decomposition.}
  \label{fig_domain1}
\end{figure}

We claim $\underline{u}^{k}_{i}|_{\partial M_{i}}$ is continuous on $\partial M_{i} \setminus (\gamma_{i} \cap S)$. Note that
\begin{equation}\label{lem_sequence_continuous_1}
\partial M_{i} \setminus S = \partial M_{i} \setminus (\gamma_{i} \cap S) = (\partial M_{i} \cap \partial M) \cup (\gamma_{i} \setminus S).
\end{equation}
Assuming $x \in \partial M_{i} \setminus S$ and $y \in \partial M_{i}$, we need to show $\underline{u}^{k}_{i} (y)$ tends to $\underline{u}^{k}_{i} (x)$ when $y$ tends to $x$. Suppose firstly $x \in \partial M_{i} \cap \partial M$. (See Fig.~\ref{fig_domain1} for an illustration, where $M$ is decomposed into three subdomains and the shadowed part is $M_{i}$.) By the boundary value condition, $\underline{u}^{k}_{i} (x) = \underline{u}^{k}_{i}|_{\partial M} = 0$. Thus $\underline{u}^{k}_{i} (y)$ tends to $0$ when $y \in \partial M_{i} \cap \partial M$ (see the left hand side of Fig.~\ref{fig_domain1}). On the other hand, $\forall y \in \gamma_{i}$ (see the right hand side of Fig.~\ref{fig_domain1}),
\begin{equation}\label{lem_sequence_continuous_2}
\underline{u}^{k}_{i} (y) = \min \{ \underline{u}^{k-1}_{j} (y) \mid y \in M_{j}, j \neq i \}.
\end{equation}
and $\underline{u}^{k-1}_{j}|_{\partial M} =0$ for all $j$. Since these $\underline{u}^{k-1}_{j}$ are continuous at $x$, we also infer $\underline{u}^{k}_{i} (y)$ tends to $0$ when $y \in \gamma_{i}$. In summary, $\underline{u}^{k}_{i}|_{\partial M_{i}}$ is continuous at $x \in \partial M_{i} \cap \partial M$. Let's assume $x \in \gamma_{i} \setminus S$ further. (See Fig.~\ref{fig_domain2} for an illustration.) There exists a neighborhood $U_{x}$ of $x$ such that, $\forall y \in U_{x} \cap \partial M_{i}$, we have $y \notin \partial M \cup S$ and $y$ belongs to the same subdomains $M_{j}$ in \eqref{lem_sequence_continuous_2} as $x$ does. By \eqref{lem_sequence_continuous_2} and the continuity of $\underline{u}^{k-1}_{j}$ at $x$ again, $\underline{u}^{k}_{i}|_{\partial M_{i}}$ is continuous at $x$. Overall, our claim is proved by \eqref{lem_sequence_continuous_1}.
\begin{figure}[htbp]
\centering
  \includegraphics[width=0.3\textwidth]{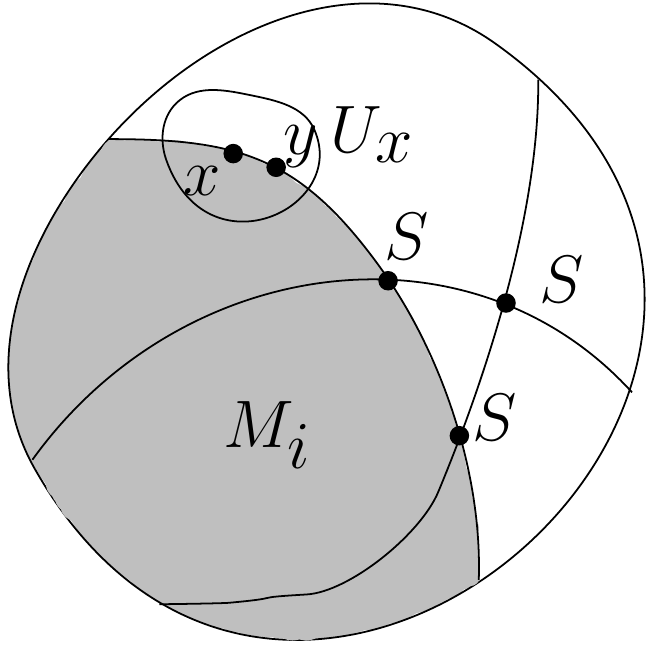}
  \caption{An illustration of domain decomposition.}
  \label{fig_domain2}
\end{figure}

Therefore, $\underline{u}^{k}_{i}$ is continuous in $M_{i} \setminus (\gamma_{i} \cap S)$ since it is a solution and $\partial M_{i}$ is Lipschitz. We also infer $\underline{u}^{k-1}_{i}|_{\partial M_{i}} \leq \underline{u}^{k}_{i}|_{\partial M_{i}}$ by the fact $\underline{u}^{k-2}_{j} \leq \underline{u}^{k-1}_{j}$ for all $j$. This further yields $\underline{u}^{k-1}_{i} \leq \underline{u}^{k}_{i}$. Since $\underline{u}^{k-1}_{j} \leq 0$ for all $j$, we also have $\underline{u}^{k}_{i}|_{\partial M_{i}} \leq 0$ which implies $\underline{u}^{k}_{i} \leq 0$. The conclusion holds for $n=k$.
\end{proof}

\begin{lemma}\label{lem_limit_continuous}
$\{ \underline{u}^{n}_{i} \}$ (resp. $\{ \bar{u}^{n}_{i} \}$) converges to a function $\underline{u}_{i} \leq 0$ (resp. $\bar{u}^{i} \geq 0$). The convergence is uniform on each $K \subset \subset M_{i} \setminus (\gamma_{i} \cap S)$. (Here ``$\subset\subset$" means $K$ is a compact subset.) Furthermore, $\underline{u}_{i}$ and $\bar{u}^{i}$ vanish on $\partial M$ and are continuous in $M_{i} \setminus (\gamma_{i} \cap S)$.
\end{lemma}
\begin{proof}
We only prove the case of $\{ \underline{u}^{n}_{i} \}$. By Lemma \ref{lem_sequence_continuous}, $\{ \underline{u}^{n}_{i} \}$ converges increasingly to a function $\underline{u}_{i} \leq 0$. Clearly, $\underline{u}_{i}|_{\partial M} =0$ since so do all $\underline{u}^{n}_{i}$. Lemma \ref{lem_sequence_continuous} indicates the continuity of $\underline{u}^{n}_{i}$ on $M_{i} \setminus (\gamma_{i} \cap S)$. By Dini's Theorem, the uniform convergence on $K$ would follow from the continuity of $\underline{u}_{i}$. So we only need to prove $\underline{u}_{i}$ is continuous on $M_{i} \setminus (\gamma_{i} \cap S)$.

As these $\underline{u}^{n}_{i}$ are solutions, by Harnack's inequality, the monotone convergence of $\{ \underline{u}^{n}_{i} \}$ is uniform on each $K' \subset \subset M_{i} \setminus \partial M_{i}$ which implies the continuity of $\underline{u}_{i}$ on $M_{i} \setminus \partial M_{i}$. It remains to show $\underline{u}_{i}$ is continuous at each $x \in \partial M_{i} \setminus (\gamma_{i} \cap S)$. By \eqref{lem_sequence_continuous_1}, we may check the continuity of $\underline{u}_{i}$ at $x \in \gamma_{i} \setminus S$ and at $x \in \partial M_{i} \cap \partial M$ separately.

We firstly claim $\underline{u}_{i}|_{\partial M_{i}}$ is continuous at each $x \in \gamma_{i} \setminus S$. (See Fig.~\ref{fig_domain2} for an illustration.) Suppose $y \in \partial M_{i}$. Since $x \notin \partial M \cup S$, so does $y$ when $y$ is close enough to $x$. In other words, $y \notin \partial M$ and $y \notin \partial M_{j}$ for all $j \neq i$. Thus, when $y \in \partial M_{i}$ is near $x$, we have $y$ belongs to the same subdomains $M_{j}$ as $x$ does and
\[
\underline{u}^{n}_{i} (y) = \min \{ \underline{u}^{n-1}_{j} (y) \mid y \in M_{j}, j \neq i \} = \min \{ \underline{u}^{n-1}_{j} (y) \mid y \in M_{j} \setminus \partial M_{j}, j \neq i \}.
\]
Since $\{ \underline{u}^{n}_{j} \}$ converges uniformly on $K' \subset \subset M_{j} \setminus \partial M_{j}$, we also have $\underline{u}^{n}_{i}|_{\partial M_{i}}$ converges uniformly near $x$. By Lemma \ref{lem_sequence_continuous}, $\underline{u}^{n}_{i}$ is continuous near $x$. Now the continuity of $\underline{u}_{i}|_{\partial M_{i}}$ at $x$ follows from that of $\underline{u}^{n}_{i}|_{\partial M_{i}}$ at $x$.

Let $\tilde{u}_{i}$ be the Perron solution on $M_{i}$ with $\tilde{u}_{i}|_{\partial M_{i}} = \underline{u}_{i}|_{\partial M_{i}}$. By the fact, $\forall n$, $\underline{u}^{n}_{i}$ is also a solution and $\underline{u}^{n}_{i}|_{\partial M_{i}} \leq \underline{u}_{i}|_{\partial M_{i}} = \tilde{u}_{i}|_{\partial M_{i}}$, we have $\underline{u}^{n}_{i} \leq \tilde{u}_{i}$ and hence $\underline{u}^{n}_{i} \leq \underline{u}_{i} \leq \tilde{u}_{i}$. Let $x \in \gamma_{i} \setminus S$. Then $\tilde{u}_{i}$ is continuous at $x$ since so is $\tilde{u}_{i}|_{\partial M_{i}}$. For each $\epsilon >0$, we have $0 \leq \tilde{u}_{i} (x) - \underline{u}^{N}_{i} (x) < \epsilon$ for some $N$ since $\{ \underline{u}^{n}_{i} (x) \}$ converges to $\tilde{u}_{i} (x)$. Furthermore, there exists a neighborhood $U_{x}$ of $x$ such that $\forall y \in U_{x} \cap M_{i}$,
\[
|\underline{u}^{N}_{i} (x) - \underline{u}^{N}_{i} (y)| < \epsilon \qquad \text{and} \qquad |\tilde{u}_{i} (x) - \tilde{u}_{i} (y)| < \epsilon
\]
because $\underline{u}^{N}_{i}$ and $\tilde{u}_{i}$ are continuous at $x$. Thus
\begin{align*}
\underline{u}_{i} (x) - 2 \epsilon & = \tilde{u}_{i} (x) - 2 \epsilon < \underline{u}^{N}_{i} (x) - \epsilon < \underline{u}^{N}_{i} (y) \\
& \leq \underline{u}_{i} (y) \leq \tilde{u}_{i} (y) < \tilde{u}_{i} (x) + \epsilon = \underline{u}_{i} (x) + \epsilon,
\end{align*}
i.e. $\underline{u}_{i} (x) - 2 \epsilon < \underline{u}_{i} (y) < \underline{u}_{i} (x) + \epsilon$. So $\underline{u}_{i}$ is continuous at each $x\in \gamma_{i} \setminus S$.

Finally, suppose $x \in \partial M_{i} \cap \partial M$. Since $\underline{u}^{0}_{i}$ is continuous, $\forall \epsilon >0$, there exists a neighborhood $U_{x}$ of $x$ such that $\forall y \in U_{x} \cap M_{i}$, $\underline{u}^{0}_{i} (x) - \epsilon < \underline{u}^{0}_{i} (y)$. We already know $\underline{u}^{0}_{i} \leq \underline{u}_{i} \leq 0$ and $\underline{u}^{0}_{i}|_{\partial M} = \underline{u}_{i}|_{\partial M} =0$. Therefore, $\underline{u}_{i} (x) = \underline{u}^{0}_{i} (x) = 0$ and
\[
- \epsilon = \underline{u}^{0}_{i} (x) - \epsilon < \underline{u}^{0}_{i} (y) \leq \underline{u}_{i} (y) \leq 0.
\]
We see $|\underline{u}_{i} (y) - \underline{u}_{i} (x)| < \epsilon$ and hence $\underline{u}_{i}$ is continuous at $x \in \partial M_{i} \cap \partial M$. We finish the proof by \eqref{lem_sequence_continuous_1}.
\end{proof}

\begin{proposition}\label{prop_sub_super}
$\{ \underline{u}^{n}_{i} \}$ and $\{ \bar{u}^{n}_{i} \}$ converge to $0$ on $M_{i} \setminus (\gamma_{i} \cap S)$. The convergence is uniform on each $K \subset \subset M_{i} \setminus (\gamma_{i} \cap S)$. Here $S$ is the set defined in Assumption \ref{asp_decomposition}.
\end{proposition}
\begin{proof}
By Lemma \ref{lem_limit_continuous}, it suffices to show $\underline{u}_{i} = \bar{u}_{i} =0$ on $M_{i} \setminus (\gamma_{i} \cap S)$. We only prove the case of $\bar{u}_{i}$ since the proof of the other case is similar.

We already know these $\bar{u}^{n}_{i}$ are solutions. By the uniform convergence on $K \subset \subset M_{i} \setminus (\gamma_{i} \cap S)$, we know that $\bar{u}_{i}$ is a solution in $M_{i} \setminus \partial M_{i}$. Furthermore, $\forall x \in \gamma_{i}$,
\[
\bar{u}^{n}_{i} (x) = \max \{ \bar{u}^{n-1}_{j} (x) \mid x \in M_{j}, j \neq i \}.
\]
We infer, $\forall x \in \gamma_{i}$,
\begin{equation}\label{prop_sub_super_1}
\bar{u}_{i} (x) = \max \{ \bar{u}_{j} (x) \mid x \in M_{j}, j \neq i \} = \max \{ \bar{u}_{j} (x) \mid x \in M_{j} \}.
\end{equation}
Define a function $w$ on $M \setminus S$ as
\begin{equation}\label{prop_sub_super_2}
w(y) = \max \{ \bar{u}_{j} (y) \mid y \in M_{j} \}.
\end{equation}
We shall prove that $w = 0$. Clearly, $w|_{\partial M} =0$ and $0 \leq w \leq \bar{u}^{0}$ since $\bar{u}_{j}|_{\partial M} =0$ and $0 \leq \bar{u}_{j} \leq \bar{u}^{0}$ for all $j$.

We firstly claim that $w$ is continuous on $M \setminus S$. Note that
\[
M \setminus S = \partial M \cup \left[ M \setminus \left( \partial M \cup \bigcup_{i=1}^{m} \gamma_{i} \right) \right] \cup \bigcup_{i=1}^{m} (\gamma_{i} \setminus S).
\]
The continuity at each $x \in \partial M$ is obvious since $\bar{u}_{j} (x) =0$ and $\bar{u}_{j}$ are continuous at $x$ for all $j$. If $x \notin \partial M \cup \bigcup_{i=1}^{m} \gamma_{i}$, there exists a neighborhood $U_{x}$ of $x$ such that, $\forall y \in U_{x}$, $y$ belongs to the same subdomains as $x$ does. The continuity of $w$ at $x$ follows from that of all $\bar{u}_{j}$. Now suppose $x \in \gamma_{i} \setminus S$. There exists a neighborhood $U_{x}$ of $x$ such that $\gamma_{i}$ divides $U_{x}$ into two sides. Furthermore, The points $y$ on one side belong to exactly one more subdomain, i.e. $M_{i}$, than the points on the other side do. (See Fig. \ref{fig_domain3} for an illustration, where $y_{1}$ and $y_{2}$ are on different sides of $\gamma_{i}$. Both $x$ and $y_{1}$ belong to $M_{i}$, while $y_{2}$ does not.) The continuity of $w$ at $x \in \gamma_{i} \setminus S$ follows from the continuity of all $\bar{u}_{j}$ together with \eqref{prop_sub_super_1}.
\begin{figure}[htbp]
\centering
  \includegraphics[width=0.3\textwidth]{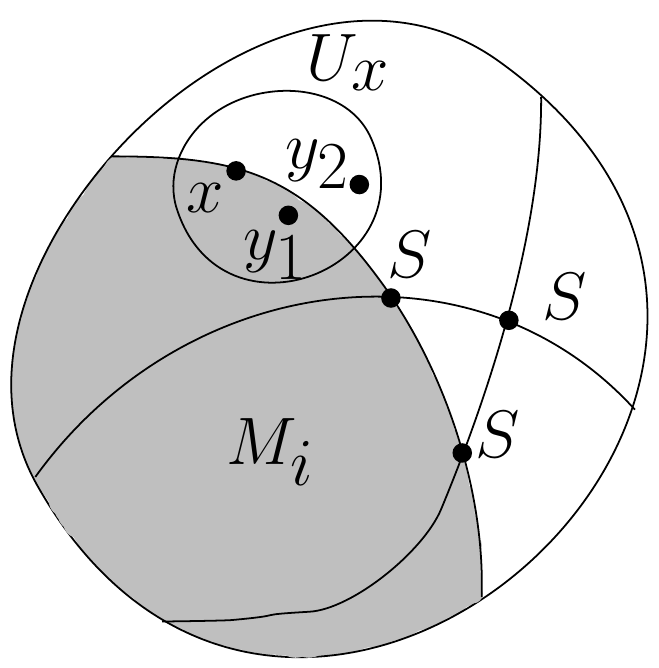}
  \caption{An illustration of domain decomposition.}
  \label{fig_domain3}
\end{figure}

Secondly, we show that $w$ is a subsolution of \eqref{eqn_homogenous} in $M \setminus S$. It suffices to check $w$ is a subsolution locally in $M \setminus (\partial M \cup S)$ since $w$ is continuous on $M \setminus S$ and $M \setminus (\partial M \cup S)$ is the interior of $M \setminus S$. This fact is obvious in $M \setminus (\partial M \cup \bigcup_{i=1}^{m} \gamma_{i})$. By \eqref{prop_sub_super_2}, $w$ is the maximum of a fixed family of $\bar{u}_{j}$ near each $x \in M \setminus (\partial M \cup \bigcup_{i=1}^{m} \gamma_{i})$. The fact follows because all $\bar{u}_{j}$ are solutions. By an argument similar to the proof of the (iii) in \cite[p.~103]{Gilbarg_Trudinger}, we also see $w$ is a subsolution near each $x \in \gamma_{i} \setminus S$. More precisely, near $x$, the family of $\bar{u}_{j}$ in \eqref{prop_sub_super_2} on one side of $\gamma_{i}$ has one more member, i.e. $\bar{u}_{i}$, than the family on the other side does. (See also Fig. \ref{fig_domain3} for an illustration.) So locally, $w$ can be viewed as a lifting of a subsolution on the first side which is again a subsolution.

By Assumption \ref{asp_decomposition}, the $(d-2)$-dimensional Hausdorff measure of $S$ is finite. We also know $w$ is bounded. Thus by \cite[Theorem~12.1~(c)]{harvey_lawson2014}, there is a canonical way (see \cite[(3.1)]{harvey_lawson2014}) to extend $w$ to be a subsolution of \eqref{eqn_homogenous} on $M$, i.e. $w|_{M \setminus \partial M}$ satisfies Definition \ref{def_sub_super} and $w$ is continuous at each $x \in \partial M$. Now we claim the extended $w \leq 0$. If $\partial M \neq \emptyset$, by the fact that $w|_{\partial M} = 0$ and $w$ is a subsolution, we see $w \leq 0$ on $M$. On the other hand, supposing $\partial M = \emptyset$, we have $b>0$ in \eqref{eqn_operator_local} by assumption. Since $w$ is upper semicontinuous and $M$ is compact, $w$ has a global maximum point $\hat{x}$. This $\hat{x}$ is in $M \setminus \partial M$ since $\partial M = \emptyset$. Taking $\underline{u} = w$ and $\psi =0$ in Definition \ref{def_sub_super}, we infer $w(\hat{x}) \leq 0$. Therefore, $w \leq 0$ on $M$.

However, $w|_{M \setminus S} \geq 0$, thus $w|_{M \setminus S} = 0$. Since $0 \leq \bar{u}^{i}|_{M_{i} \setminus S} \leq w|_{M_{i} \setminus S}$, we obtain $\bar{u}^{i}|_{M_{i} \setminus S} = 0$. Thus $\bar{u}^{i}|_{M_{i} \setminus (\gamma_{i} \cap S)} =0$ by the continuity of $\bar{u}^{i}|_{M_{i} \setminus (\gamma_{i} \cap S)}$.
\end{proof}

\subsection{Proof of Convergence}
Again, in the following, all arguments cover both the case $\partial M \neq \emptyset$ and $\partial M = \emptyset$. All statements related to $\partial M$ are vacuously true if $\partial M = \emptyset$. In addition, $C^{0}_{0} (M) = C^{0} (M)$ if $\partial M = \emptyset$.

\begin{lemma}\label{lem_barrier}
Let $u^{0}$ be the initial guess in Algorithm \ref{alg_continuous}. Suppose further $u$ and $u^{0}$ are Lipschitz continuous if $\partial M \neq \emptyset$. Then there exist a subsolution $\underline{u}^{0} \in C^{0}_{0} (M)$ and a supersolution $\bar{u}^{0} \in C^{0}_{0} (M)$ of \eqref{eqn_homogenous} such that $\underline{u}^{0} \leq u^{0} -u \leq \bar{u}^{0}$.
\end{lemma}
\begin{proof}
Firstly, let's assume $\partial M \neq \emptyset$. Clearly, the boundary value problem
\[
\left\{
\begin{aligned}
Lv & =1, & \text{in $M$}, \\
v & =0, & \text{on $\partial M$},
\end{aligned}
\right.
\]
has a unique solution $v \in C^{\infty} (M)$. By Hopf's Lemma and strong maximum principle (\cite[\S~3.2]{Gilbarg_Trudinger}), we have $\frac{\partial v}{\partial \nu} <0$ and $v|_{M \setminus \partial M} >0$, where $\frac{\partial v}{\partial \nu}$ is the outer normal derivative of $v$ on $\partial M$. Since $u^{0} -u$ is Lipschitz and $(u^{0} -u)|_{\partial M} =0$, there exists an open neighborhood $V$ of $\partial M$ such that $|u^{0} - u| \leq k v$ in $V$ for sufficiently large constant $k$. Furthermore, $v|_{M \setminus V}$ has a positive minimum. We see $|u^{0} - u| \leq k v$ on the whole $M$ for sufficiently large $k$. We can define $\underline{u}^{0} = -kv$ and $\bar{u}^{0} = kv$ for some large $k$.

On the other hand, suppose $\partial M = \emptyset$, then we have $b>0$ in \eqref{eqn_problem} as assumed. Since $u^{0} - u$ is continuous and hence bounded on $M$. We can easily define $\underline{u}^{0} = -k$ and $\bar{u}^{0} = k$ for some large constant $k$.
\end{proof}

\begin{lemma}\label{lem_sub_super}
Let $\{ \underline{u}^{n}_{i} \}$ (resp. $\{ \bar{u}^{n}_{i} \}$) be the sequence generated in \eqref{eqn_DDM_sub} (resp. \eqref{eqn_DDM_sup}) with $\underline{u}^{0}$ (resp. $\bar{u}^{0}$) satisfying the conclusion of Lemma \ref{lem_barrier}. Then, $\forall n$, $\forall i$, $\underline{u}^{n}_{i} \leq u^{n}_{i} -u \leq \bar{u}^{n}_{i}$.
\end{lemma}
\begin{proof}
We only prove $\underline{u}^{n}_{i} \leq u^{n}_{i} -u$ because the proof of $u^{n}_{i} -u \leq \bar{u}^{n}_{i}$ is similar. Since $\underline{u}^{0}_{i} = \underline{u}^{0}|_{M_{i}}$ and $u^{0}_{i} = u^{0}|_{M_{i}}$, the conclusion is true for $n=0$.

Suppose the conclusion is true for $n-1$. Since $L \underline{u}^{n}_{i} = L (u^{n}_{i} -u) =0$, it suffices to prove that, $\forall x \in \partial M_{i}$, $\underline{u}^{n}_{i} (x) \leq u^{n}_{i} (x) - u(x)$. This is certainly true for $x \in \partial M$ since $\underline{u}^{n}_{i} (x) = u^{n}_{i} (x) - u(x) =0$. For $x \in \partial M_{i} \setminus \partial M = \gamma_{i}$, by inductive hypothesis,
\begin{align*}
\underline{u}^{n}_{i} (x) & = \min \{ \underline{u}^{n-1}_{j} (x) \mid x \in M_{j}, j \neq i \} \leq \min \{ u^{n-1}_{j} (x) - u(x) \mid x \in M_{j}, j \neq i \} \\
& \leq \sum_{j=1}^{m} \rho_{j} (x) [u^{n-1}_{j} (x) - u(x)] = u^{n}_{i} (x) - u(x).
\end{align*}
Note that we have used the fact $\rho_{i} (x) =0$ when $x \in \gamma_{i}$ and $\rho_{j} (x) =0$ when $x \notin M_{j}$ in the last inequality. The proof is finished.
\end{proof}

Now we are ready to prove Theorem \ref{thm_convergence}.
\begin{proof}[Proof of Theorem \ref{thm_convergence}]
(1) It suffices to prove $\lim\limits_{n \rightarrow \infty} \| u^{n}_{i} - u \|_{C^{0} (M_{i})} =0$ for each $i$. Since $L(u^{n}_{i} - u) =0$, by the maximum principle, we only need to show that $\{ (u^{n}_{i} -u)|_{\partial M_{i}} \}$ converges uniformly to $0$. Let $\{ \underline{u}^{n}_{j} \}$ and $\{ \bar{u}^{n}_{j} \}$ be the sequences stated in Lemma \ref{lem_sub_super}. By Assumption \ref{asp_partition}, we have $\mathrm{supp} \rho_{j} \subset \subset M_{j} \setminus \overline{\gamma_{j}}$. By Proposition \ref{prop_sub_super}, $\forall j$, $\{ \underline{u}^{n}_{j} \}$ and $\{ \bar{u}^{n}_{j} \}$ converge uniformly to $0$ on $\mathrm{supp} \rho_{j}$. By Lemma \ref{lem_sub_super}, so does $\{ u^{n}_{j} -u \}$. Since $(u^{n}_{i} -u)|_{\partial M_{i}} = \sum_{j=1}^{m} \rho_{j} (u^{n-1}_{j} -u)$, we infer $\{ (u^{n}_{i} -u)|_{\partial M_{i}} \}$ converges uniformly to $0$.

(2) By the fact $u^{n} - u = \sum_{i=1}^{m} \rho_{i} (u^{n}_{i} - u)$, it's obvious to see $\| u^{n} - u \|_{C^{0}(M)} \leq \theta^{n}$ and $\| u^{n} - u \|_{L^{2}(M)} \leq C_{3} \theta^{n}$ for some constant $C_{3}$ independent of $n$ and $i$.

Secondly, we claim that $\forall K_{i} \subset \subset M_{i} \setminus \overline{\gamma_{i}}$, there exists a constant $C_{4} (K_{i})$ independent of $n$ such that, $\forall n$,
\begin{equation}\label{thm_convergence_1}
\| \nabla u^{n}_{i} - \nabla u \|_{L^{2} (K_{i})} \leq C_{4} (K_{i}) \| u^{n}_{i} - u \|_{C^{0}(M_{i})}.
\end{equation}
Since $L(u^{n}_{i} - u) =0$ in $M_{i} \setminus \partial M_{i}$ and $(u^{n}_{i} - u)|_{\partial M} =0$, the claim \eqref{thm_convergence_1} follows from the regularity theory for elliptic equations (cf. \cite[p.~331]{evans} and \cite[p.~340]{evans}).

Furthermore, since $\rho_{i}$ is Lipschitz continuous, we have $\rho_{i} \in W^{1,\infty} (M)$ (see Theorem 5 in \cite[\S4.2.3]{evans_gariepy}) and
\[
\nabla (u^{n} -u) = \sum_{i=1}^{m} \nabla \rho_{i} \cdot (u^{n}_{i} -u) + \sum_{i=1}^{m} \rho_{i} \cdot \nabla (u^{n}_{i} -u),
\]
which implies
\[
\| \nabla (u^{n} -u) \|_{L^{2}(M)} \leq \sum_{i=1}^{m} \| \nabla \rho_{i} \|_{L^{\infty}(M_{i})} \cdot \| u^{n}_{i} -u \|_{L^{2}(M_{i})} + \sum_{i=1}^{m} \| \nabla u^{n}_{i} - \nabla u \|_{L^{2}(\mathrm{supp}\rho_{i})}.
\]
Clearly, $\| u^{n}_{i} -u \|_{L^{2}(M_{i})} \leq C_{5} \theta^{n}$ for some constant $C_{5}$ independent of $n$ and $i$. By \eqref{thm_convergence_1} and the fact $\mathrm{supp}\rho_{i} \subset \subset M_{i} \setminus \overline{\gamma_{i}}$, we also have
\[
\| \nabla u^{n}_{i} - \nabla u \|_{L^{2}(\mathrm{supp}\rho_{i})} \leq C_{4} (\mathrm{supp}\rho_{i}) \cdot \theta^{n}.
\]
Now the conclusion $\| u^{n} - u \|_{H^{1}(M)} \leq C_{1} \theta^{n}$ follows.

(3) As $L(u^{n}_{i} -u) =0$ and $u^{n}_{i}|_{\partial M_{i}} = u^{n-1}|_{\partial M_{i}}$, we have $u^{n}_{i} -u^{n-1} \in H^{1}_{0} (M_{i})$ and
\[
L(u^{n}_{i} -u^{n-1})= L(u^{n}_{i} -u) + L(u - u^{n-1}) = L(u - u^{n-1}).
\]
By the Poincar\'{e} inequality and the fact $b \geq 0$ in \eqref{eqn_problem}, there exist positive constants $C_{6}$ and $C_{7}$ independent of $n$ and $i$ such that
\begin{align*}
& C_{6} \| u^{n}_{i} -u^{n-1} \|^{2}_{H^{1}(M_{i})} \leq (L(u^{n}_{i} -u^{n-1}), u^{n}_{i} -u^{n-1})_{M_{i}} \\
= & (L(u - u^{n-1}), u^{n}_{i} -u^{n-1})_{M_{i}}
\leq \| L(u^{n-1} -u) \|_{H^{-1}(M_{i})} \cdot \| u^{n}_{i} -u^{n-1} \|_{H^{1}(M_{i})} \\
\leq & C_{7} \| u^{n-1} -u \|_{H^{1}(M_{i})} \cdot \| u^{n}_{i} -u^{n-1} \|_{H^{1}(M_{i})},
\end{align*}
where $(\cdot, \cdot)_{M_{i}}$ is the dual pairing between $H^{-1}(M_{i})$ and $H^{1}_{0}(M_{i})$. Thus
\[
\| u^{n}_{i} -u^{n-1} \|_{H^{1}(M_{i})} \leq C_{6}^{-1} C_{7} \| u^{n-1} -u \|_{H^{1}(M_{i})}
\]
and
\[
\| u^{n}_{i} -u \|_{H^{1}(M_{i})} \leq (1+ C_{6}^{-1} C_{7}) \| u^{n-1} -u \|_{H^{1}(M_{i})}.
\]
Now $\| u^{n}_{i} -u \|_{H^{1}(M_{i})} \leq C_{2} \theta^{n-1}$ follows from the fact $\| u^{n-1} - u \|_{H^{1}(M)} \leq C_{1} \theta^{n-1}$.
\end{proof}

\section{Numerical Scheme}\label{sec_scheme}
In this section, we propose a numerical DDM iteration procedure (Algorithm \ref{alg_discrete} below) to obtain approximations to the solution to (\ref{eqn_problem}). The idea is as follows. We decompose $M$ into overlapping subdomains $M_{i}$ satisfying Assumption \ref{asp_decomposition}. Furthermore, each $M_{i}$ is required in a coordinate chart. Then apply a DDM iteration procedure similar to Algorithm \ref{alg_continuous}. Since $M_{i}$ is in a coordinate chart, an elliptic problem on $M_{i}$ is converted to one on a domain in a Euclidean space and then can be solved approximately by usual finite element methods. The transition of information among subdomains is by interpolation.

Unlike the numerical DDM procedure in \cite[Algorithm~3.1]{cao_qin}, the iteration of this algorithm is highly parallel. As a result, this algorithm is of particular interest when the number of subdomains is large.

Overall, our numerical Algorithm \ref{alg_discrete} is an imitation of Algorithm \ref{alg_continuous} whose convergence has been justified by the main Theorem \ref{thm_convergence}. However, the convergence theory of the numerical scheme is not established in this paper.
The current work has two objectives. First, we provide a numerical scheme and numerical tests to assess Theorem \ref{thm_convergence}. Second, we test the feasibility and efficacy of our numerical scheme by numerical experiments.

\subsection{Finite Element Spaces over a \texorpdfstring{d}{$d$}-Rectangle}
Our numerical algorithm is based on finite element methods. Suppose a manifold $M$ has dimension $d$. As indicated above, each subdomain $M_{i}$ of $M$ will be converted to a domain $D_{i} \subset \mathbb{R}^{d}$. To minimize the difficulty of coding, as in \cite{cao_qin}, we shall choose $D_{i}$ as a $d$-rectangle and use the multilinear finite element space of $d$-rectangles, i.e. the space of $d$-rectangle of type ($1$) in \cite[p.~56-64]{ciarlet}.

Recall that a $d$-rectangle $D$ is
\[
D = \prod_{i=1}^{d} [a_{i}, b_{i}] = \{ (x_{1}, \cdots, x_{d}) \mid \forall i, x_{i} \in [a_{i}, b_{i}] \}.
\]
We refine each coordinate factor interval $[a_{i}, b_{i}]$ by adding points of partition:
\[
a_{i} = c_{i,0} < c_{i,1} < \cdots < c_{i,N_{i}} = b_{i}.
\]
Then $[a_{i}, b_{i}]$ is divided into $N_{i}$ subintervals. Define a function $\varphi_{i,j}$ on $[a_{i}, b_{i}]$ for $0 \leq j \leq N_{i}$ as
\begin{equation}\label{eqn_1d_base}
\varphi_{i,j} (x_{i}) =
\begin{cases}
\frac{x_{i} - c_{i,j-1}}{c_{i,j} - c_{i,j-1}}, & x_{i} \in [c_{i,j-1}, c_{i,j}]; \\
\frac{x_{i} - c_{i,j+1}}{c_{i,j} - c_{i,j+1}}, & x_{i} \in [c_{i,j}, c_{i,j+1}]; \\
0, & \text{otherwise}.
\end{cases}
\end{equation}
Here $\varphi_{i,j} (x_{i})$ is undefined for $x_{i} < c_{j}$ (resp. $x_{i} > c_{j}$) when $j=0$ (resp. $j= N_{i}$). Clearly, $\varphi_{i,j}$ is piecewise linear such that $\varphi_{i,j} (c_{i,j}) =1$ and $\varphi_{i,j} (c_{i,t}) =0$ for $t \neq j$.

The refinement of all such $[a_{i}, b_{i}]$ provides a grid on $D$. This divides $D$ into $\prod_{i=1}^{d} N_{i}$ many small $d$-rectangles
\begin{equation}\label{eqn_element}
\prod_{i=1}^{d} [c_{i,t_{i}-1}, c_{i,t_{i}}],
\end{equation}
where $1 \leq t_{i} \leq N_{i}$ for all $i$. Each small $d$-rectangle \eqref{eqn_element} is an element of the grid. A vertex of \eqref{eqn_element} is a node of the grid which is of the form
\[
\xi = (c_{1, j_{1}}, c_{2, j_{2}}, \cdots, c_{d, j_{d}}),
\]
where $0 \leq j_{i} \leq N_{i}$ for all $i$. Let $W_{h}$ be the finite element space of $d$-rectangles of type ($1$) (see \cite[p.~57]{ciarlet}). A base function in $W_{h}$ associated with the node $\xi$ is
\[
\varphi_{\xi} (x_{1}, \cdots, x_{d}) = \prod_{i=1}^{d} \varphi_{i,j_{i}} (x_{i}),
\]
where $\varphi_{i,j_{i}}$ is the one in \eqref{eqn_1d_base}.

\subsection{Numerical Algorithm}
Let $M$ be a $d$-dimensional compact Riemannian manifold with or without boundary. We shall propose a numerical DDM based on finite element methods to find numerical solutions to \eqref{eqn_problem}

Suppose $M = \bigcup_{i=1}^{m} (M_{i} \setminus \overline{\gamma_{i}})$ as Assumption \ref{asp_decomposition}. Assume further there is a smooth diffeomorphism $\phi_{i}: D_{i} \rightarrow M_{i}$ for each $i$, where $D_{i}$ is a $d$-rectangle in $\mathbb{R}^{d}$. Theoretically, we can always get such triples $(M_{i}, D_{i}, \phi_{i})$. The reason has been explained in \cite[\S~3.2]{cao_qin} provided that $\partial M = \emptyset$. On the contrary, if $\partial M \neq \emptyset$, special technique should be employed to deal with the boundary. We shall explain this technique in Section \ref{sec_boundary}.

We also need a partition of unity $\{ \rho_{i} \mid 1 \leq i \leq m \}$ satisfying Assumption \ref{asp_partition}. Such $\rho_{i}$ can be constructed as follows. Choose nonnegative Lipschitz continuous functions $\sigma_{i}$ on $M$ such that $\mathrm{supp} \sigma_{i} \subset M_{i}  \setminus \overline{\gamma_{i}}$ and $\sum_{i=1}^{m} \sigma_{i} >0$ on $M$. We define desired $\rho_{i}$ as
\begin{equation}\label{eqn_rho}
\rho_{i} = \frac{\sigma_{i}}{\sum_{j=1}^{m} \sigma_{j}}.
\end{equation}
To obtain $\sigma_{i}$, it suffices to define $\sigma_{i} \circ \phi_{i}$ which is a function on $D_{i}$ and hence can be defined simply and quite arbitrarily in terms of elementary functions.

Our numerical algorithm is a discrete imitation of Algorithm \ref{alg_continuous}. Since the numerical procedure is based on finite element methods, it's necessary to convert the first line of \eqref{alg_continuous_1} to its weak form: $\forall v \in H^{1}_{0} (M_{i})$,
\[
\int_{M_{i}} (\langle \nabla u^{n}_{i}, \nabla v \rangle + bu^{n}_{i} v)\, \mathrm{dvol} = \int_{M_{i}} fv \,\mathrm{dvol}.
\]
Here $\nabla u^{n}_{i}$ and $\nabla v$ are the gradients of $u^{n}_{i}$ and $v$ with respect to the Riemannian metric $g$, $\langle \nabla u^{n}_{i}, \nabla v \rangle$ is the inner product of $\nabla u^{n}_{i}$ and $\nabla v$, and $\mathrm{dvol}$ is the volume form. In terms of the local coordinates $(x_{1}, \dots, x_{d})$ provided by $\phi_{i}: D_{i} \rightarrow M_{i}$,
\[
\langle \nabla  u^{n}_{i}, \nabla  v \rangle = \sum_{\alpha,\beta=1}^{d} g^{\alpha \beta} \frac{\partial u^{n}_{i}}{\partial x_{\alpha}} \frac{\partial v}{\partial x_{\beta}}
\qquad
\text{and}
\qquad
\mathrm{dvol} = \sqrt{G} \mathrm{d} x_{1} \cdots \mathrm{d} x_{d}.
\]
Thus we define an energy  bilinear form on $H^{1} (D_{i})$ as
\[
a_{i} (w,v) = \int_{D_{i}} \left( \sum_{\alpha,\beta=1}^{d} g^{\alpha \beta} \frac{\partial w}{\partial x_{\alpha}} \frac{\partial v}{\partial x_{\beta}} + bwv \right) \sqrt{G} \mathrm{d} x_{1} \cdots \mathrm{d} x_{d}.
\]
Define a bilinear form $(\cdot, \cdot)_{i}$ on $L^{2} (D_{i})$ as
\[
(w,v)_{i} = \int_{D_{i}} wv \sqrt{G} \mathrm{d} x_{1} \cdots \mathrm{d} x_{d}.
\]
The first line of \eqref{alg_continuous_1} is converted to an equation on $D_{i}$ in weak form: $\forall v \in H^{1}_{0} (D_{i})$,
\begin{equation}\label{eqn_problem_coordinate}
a_{i} (u^{n}_{i} \circ \phi_{i}, v) = (f \circ \phi_{i}, v)_{i}.
\end{equation}

Create a grid of $d$-rectangles over $D_{i}$. Let $V_{h,i}$ be the finite element space of $d$-rectangles of type $(1)$ over $D_{i}$. A discrete imitation of the first line of \eqref{alg_continuous_1} would be: find a $u_{h,i}^{n} \in V_{h,i}$ such that, $\forall v_{h} \in V_{h,i} \cap H^{1}_{0} (D_{i})$,
\[
a_{i} (u_{h,i}^{n}, v_{h}) = (f \circ \phi_{i}, v_{h})_{i}.
\]
However, this discrete problem is not well-posed because the degrees of freedom of $u_{h,i}^{n}$ on $\partial D_{i}$ are undetermined. As an imitation of the second line of \eqref{alg_continuous_1}, we should evaluate these degrees of freedom by the data in $d$-rectangles $D_{j}$ for $j \neq i$. So we have to investigate the transitions of coordinates.

For $i \neq j$, let $D_{ij} = \phi_{i}^{-1} (M_{i} \cap M_{j}) \subseteq D_{i}$ and $D_{ji} = \phi_{j}^{-1} (M_{i} \cap M_{j}) \subseteq D_{j}$. Then
\[
\phi_{j}^{-1} \circ \phi_{i}: \ D_{ij} \rightarrow D_{ji}
\]
is a diffeomorphism which is the transition of coordinates on the overlap between $M_{i}$ and $M_{j}$. We refer to \cite[\S~3.2]{cao_qin} for a detailed explanation of the transition maps. (Particularly, Fig.~1 in \cite{cao_qin} provides an illustration.) We still wish to point out that $\phi_{j}^{-1} \circ \phi_{i}$ preserves neither nodes nor grid necessarily. In other words, $\phi_{j}^{-1} \circ \phi_{i}$ may neither map a node in $D_{ij}$ to a node in $D_{ji}$, nor map the grid over $D_{ij}$ to the one over $D_{ji}$. As emphasized in \cite[\S~3.2]{cao_qin}, this incompatibility among the grids over different $D_{i}$ actually shows the high flexibility of our approach.

Now we are in a position to propose our numerical Algorithm \ref{alg_discrete}. Define
\[
V_{h} = \bigoplus_{i=1}^{m} V_{h,i}.
\]

\begin{algorithm}[ht]
\caption{A numerical DDM to solve \eqref{eqn_problem}.}
\label{alg_discrete}
\begin{algorithmic}[1]
\State Choose an arbitrary initial guess $u_{h}^{0} = (u_{h,1}^{0}, \cdots, u_{h,m}^{0}) \in V_{h}$ such that, $\forall i$, $u_{h,i}^{0} (\xi) = \varphi (\phi_{i} (\xi))$ if $\xi \in \partial D_{i}$ is a node and $\phi_{i} (\xi) \in \partial M$.

\State For each $n>0$, assuming $u_{h}^{n-1}$ has been obtained, for $1 \leq i \leq m$, find a $u_{h,i}^{n} \in V_{h,i}$ as follows. Suppose $\xi \in \partial D_{i}$ is a node. If $\phi_{i} (\xi) \in \partial M$, let $u_{h,i}^{n} (\xi) = \varphi (\phi_{i} (\xi))$. Otherwise, define
\begin{equation}\label{alg_discrete_1}
u_{h,i}^{n} (\xi) = \sum_{j=1}^{m} \rho_{j} (\phi_{i} (\xi)) \cdot u_{h,j}^{n-1} (\phi_{j}^{-1} \circ \phi_{i} (\xi)).
\end{equation}
The interior degrees of freedom of $u_{h,i}^{n}$ are determined by $\forall v_{h} \in V_{h,i} \cap H^{1}_{0} (D_{i})$,
\[
a_{i} (u_{h,i}^{n}, v_{h}) = (f \circ \phi_{i}, v_{h})_{i}.
\]

\State Define $u_{h}^{n} = (u_{h,1}^{n}, \cdots, u_{h,m}^{n}) \in V_{h}$.

\end{algorithmic}
\end{algorithm}

The computation in \eqref{alg_discrete_1} needs more explanation. If $\phi_{i} (\xi) \notin M_{j}$, then $\rho_{j} (\phi_{i} (\xi)) =0$ and the summand $\rho_{j} (\phi_{i} (\xi)) \cdot u_{h,j}^{n-1} (\phi_{j}^{-1} \circ \phi_{i} (\xi))$ is actually not needed. Otherwise, $\phi_{i} (\xi) \in M_{j}$ and
\[
\rho_{j} (\phi_{i} (\xi)) = (\rho_{j} \circ \phi_{j}) (\phi_{j}^{-1} \circ \phi_{i} (\xi)).
\]
Here $\rho_{j} \circ \phi_{j}$ is a function on $D_{j}$ with an explicit formula. Thus $\rho_{j} (\phi_{i} (\xi))$ can be obtained once we get the coordinates of $\phi_{j}^{-1} \circ \phi_{i} (\xi)$. Similarly, $u_{h,j}^{n-1} (\phi_{j}^{-1} \circ \phi_{i} (\xi))$ can also be evaluated. Since $\phi_{j}^{-1} \circ \phi_{i} (\xi)$ is not necessarily a node, the evaluation $u_{h,j}^{n-1} (\phi_{j}^{-1} \circ \phi_{i} (\xi))$ is essentially an interpolation of the degrees of freedom of $u_{h,j}^{n-1}$.

Now the $u_{h}^{n}$ in Algorithm \ref{alg_discrete} is the $n$th iterated numerical approximation to the solution to \eqref{eqn_problem}. Unlike the $u^{n}$ in Algorithm \ref{alg_continuous} which is globally defined on $M$, the $u_{h}^{n}$ consists of a group of numerical solutions $u_{h,i}^{n}$ on $D_{i}$ for $1 \leq i \leq m$. In fact, the $u_{h,i}^{n}$ in Algorithm \ref{alg_discrete} is a numerical imitation of the $u_{i}^{n}$ in Algorithm \ref{alg_continuous}. As suggested by Theorem \ref{thm_convergence}, $u_{h,i}^{n}$ should approximate $u \circ \phi_{i}$ well on $D_{i}$ when the grid scale is small and $n$ is large. Our numerical experiments show that it is indeed the case.

It's worth emphasizing that, for a fixed $n$, these $u_{h,i}^{n}$ ($1 \leq i \leq m$) in Algorithm \ref{alg_discrete} can be solved independently. Thus Algorithm \ref{alg_discrete} is a highly parallel iterative procedure, which is remarkably different from that numerical method in \cite{cao_qin}. The current method would have more advantages when the number of subdomains increases.

\section{Manifolds with Boundary}\label{sec_boundary}
We explain a technique to decompose $M$ as $M = \bigcup_{i=1}^{m} (M_{i} \setminus \overline{\gamma_{i}})$ such that $M_{i}$ can be parameterized by $d$-rectangles. For the case of $\partial M = \emptyset$, this has been explained in \cite[\S~3.2]{cao_qin}.

Now suppose $\partial M  \neq \emptyset$. The difficulty occurs only in the part near the boundary. But this difficulty can be indeed overcome theoretically. By the Collar Neighborhood Theorem (see Theorem 6.1 in \cite[Chapter~4]{hirsch}), there exists a smooth diffeomorphism
\[
\psi:\ [\delta, \eta] \times \partial M \rightarrow V \subset M
\]
such that $V$ is a neighborhood of $\partial M$ in $M$. Such a $V$ is called a \textit{collar neighborhood} of $\partial M$. One may even arrange $\psi (\{ \delta \} \times \partial M) = \partial M$ or $\psi (\{ \eta \} \times \partial M) = \partial M$ as one wishes. Let's assume $\psi (\{ \eta \} \times \partial M) = \partial M$.

For each $\zeta \in M \setminus \partial M$, there is an open chart neighborhood $U_{\zeta} \subseteq M \setminus \partial M$ of $\zeta$, i.e. there is a diffeomorphism $\phi_{\zeta}: \Omega_{\zeta} \rightarrow U_{\zeta}$, where $\Omega_{\zeta}$ is an open subset of $\mathbb{R}^{d}$. Since $\phi_{\zeta}^{-1} (\zeta)$ is an interior point of $\Omega_{\zeta}$, we can choose a rectangular neighborhood $D_{\zeta}$ of $\phi_{\zeta}^{-1} (\zeta)$ such that $D_{\zeta} \subseteq \Omega_{\zeta}$. This yields a diffeomorphism $\phi_{\zeta}: D_{\zeta} \rightarrow M_{\zeta} \subset U_{\zeta}$, where $M_{\zeta}$ is a neighborhood of $\zeta$. Now
\[
\{ M_{\zeta} \setminus \partial M_{\zeta} \mid \zeta \in M \setminus \partial M \} \cup \{ \psi ((\delta, \eta] \times \partial M) \}
\]
is an open covering of $M$. By the compactness of $M$, we obtain a finite subcovering
\[
\{ M_{\zeta_{i}} \setminus \partial M_{\zeta_{i}} \mid 1 \leq i \leq m_{1} \} \cup \{ \psi ((\delta, \eta] \times \partial M) \}.
\]
Obviously, $\psi ((\delta, \eta] \times \partial M)$ appears in the subcovering necessarily since these $M_{\zeta_{i}} \setminus \partial M_{\zeta_{i}}$ are contained in $M \setminus \partial M$. For brevity, let $(M_{i}, D_{i}, \phi_{i})$ denote $(M_{\zeta_{i}}, D_{\zeta_{i}}, \phi_{\zeta_{i}})$. Since $M_{i} \subset M \setminus \partial M$, we have $\partial M_{i} = \partial M_{i} \setminus \partial M = \gamma_{i}$.

It remains to decompose $V = \psi ([\delta, \eta] \times \partial M)$ into subdomains which can be parameterized by rectangles. Let's decompose $[\delta, \eta] \times \partial M$ firstly. Since $\partial M$ is a manifold without boundary, we can find a desired decomposition of $\partial M$. In other words, there are triples $\{ (M'_{j}, D'_{j}, \phi'_{j}) \mid 1 \leq j \leq m_{2} \}$ such that $D'_{j}$ are $(d-1)$-rectangles and $\partial M = \cup_{j=1}^{m_{2}} (M'_{j} \setminus \partial M'_{j})$. Note that $[\delta, \eta] \times \partial M$ is a product manifold. As pointed out in \cite[\S~4]{cao_qin}, decompositions of factor manifolds canonically result in a decomposition of their product. The situation of $[\delta, \eta] \times \partial M$ is even simpler because $[\delta, \eta]$ is already a $1$-rectangle. Let $D_{m_{1} + j} = [\delta, \eta] \times D'_{j}$. Then all
\[
\mathrm{id} \times \phi'_{j}:\ D_{m_{1} + j} = [\delta, \eta] \times D'_{j} \rightarrow [\delta, \eta] \times M'_{j} \subset [\delta, \eta] \times \partial M
\]
provide a desired decomposition of $[\delta, \eta] \times \partial M$, where $\mathrm{id}$ is the identity on $[\delta, \eta]$.

Finally, define
\[
\phi_{m_{1} +j} = \psi \circ (\mathrm{id} \times \phi'_{j}):\ D_{m_{1} + j} \rightarrow M.
\]
Let $M_{m_{1} +j} = \phi_{m_{1} +j} (D_{m_{1} + j})$ and $\gamma_{m_{1} +j} = \partial M_{m_{1} +j} \setminus \partial M$. Then
\[
M_{m_{1} +j} \setminus \overline{\gamma_{m_{1} +j}} = \phi_{m_{1} +j} ((\delta, \eta] \times (D_{m_{1} + j} \setminus \partial D_{m_{1} + j})).
\]
It's easy to see that $\{ (M_{i}, D_{i}, \phi_{i}) \mid 1 \leq i \leq m_{1} + m_{2} \}$ is a decomposition of $M$ such that $M = \bigcup_{i=1}^{m_{1} +m_{2}} (M_{i} \setminus \overline{\gamma_{i}})$.

We would like to mention that the transition map between $D_{m_{1} + j}$ and $D_{m_{1} + k}$ has a simple expression. In fact,
\begin{align*}
& \phi_{m_{1} + k}^{-1} \circ \phi_{m_{1} + j} = (\psi \circ (\mathrm{id} \times \phi'_{k}))^{-1} \circ (\psi \circ (\mathrm{id} \times \phi'_{j})) \\
= & (\mathrm{id} \times \phi'_{k})^{-1} \circ \psi^{-1} \circ \psi \circ (\mathrm{id} \times \phi'_{j}) = (\mathrm{id} \times \phi'_{k})^{-1} \circ (\mathrm{id} \times \phi'_{j}) = \mathrm{id} \times ({\phi'_{k}}^{-1} \circ \phi'_{j}).
\end{align*}
Here ${\phi'_{k}}^{-1} \circ \phi'_{j}$ is exactly a transition map for the decomposition of $\partial M$.

To obtain a partition of unity, by \eqref{eqn_rho}, it suffices to get certain $\sigma_{i}$. We can construct a $\sigma_{m_{1} +j}$ on $M_{m_{1} +j}$ as follows. Firstly, choose a nonnegative function $\sigma'_{j}$ on $M'_{j}$ such that $\mathrm{supp} \sigma'_{j} \subset M'_{j} \setminus \partial M'_{j}$. Note that $\sigma'_{j}$ is defined on a subdomian $M'_{j}$ of $\partial M$, and hence it is a function of $(d-1)$ variables which can be expressed locally by $(d-1)$ coordinates. Choose a nonnegative $1$-variable function $\sigma''$ on $[\delta, \eta]$ such that $\mathrm{supp} \sigma'' = [\delta'', \eta]$, where $\delta < \delta'' < \eta$. Then $\sigma'' \times \sigma'_{j}$ is a function on $[\delta, \eta] \times M'_{j} \subset [\delta, \eta] \times \partial M$. Define $\sigma_{m_{1} +j} = (\sigma'' \times \sigma'_{j}) \circ \psi^{-1}$. We see $\sigma_{m_{1} +j}$ is a nonnegative function on $M_{m_{1} +j}$ with $\mathrm{supp} \sigma_{m_{1} +j} \subset M_{m_{1} +j} \setminus \overline{\gamma_{m_{1} +j}}$.

\section{Numerical Experiments}\label{sec_experiment}
We perform several numerical tests of Algorithm \ref{alg_discrete}. The proposed method certainly applies to problems in all dimensions. However, the sizes of the linear systems derived from subdomains will increase exponentially with respect to the dimension. This difficulty is so called ``the curse of dimensionality''. Due to the constraint of computing resources, just as in \cite{cao_qin}, we only deal with problems with dimension no more than $4$. Nevertheless, the numerical examples in \cite{cao_qin} are all on manifolds without boundary. We now deal with manifolds with boundary too. Our examples of manifolds are $\mathbb{CP}^{2}$, $B^{4}$ and $B^{2} \times S^{2}$. All of them are of dimension $4$, the $\mathbb{CP}^{2}$ is without boundary, and the latter two are with boundary.

\subsection{A Problem on \texorpdfstring{$\mathbb{CP}^{2}$}{CP2}}
Let $M = \mathbb{CP}^{2}$ be the complex projective plane. It is a compact complex manifold with complex dimension $2$. Certainly, it can be considered as a real manifold with dimension $4$.

In this example, the definition of $\mathbb{CP}^{2}$, the Riemannian metric, the concrete problem of \eqref{eqn_problem}, the decomposition of $\mathbb{CP}^{2}$, and the transition maps are the same as those in \cite[\S~5.2]{cao_qin}. We refer to \cite{cao_qin} for a thorough and detailed description.

Particularly, $\mathbb{CP}^{2}$ is decomposed into $3$ subdomains $M_{j}$ for $0 \leq j \leq 2$, which is provided by diffeomorphisms $\phi_{j}: D_{j} \rightarrow M_{j}$. We have $D_{j} = [-r, r]^{4} \subset \mathbb{C}^{2} = \mathbb{R}^{4}$,
\[
\phi_{0} (z_{1}, z_{2}) = [1, z_{1}, z_{2}], \quad \phi_{1} (z_{0}, z_{2}) = [z_{0}, 1, z_{2}], \quad \phi_{2} (z_{0}, z_{1}) = [z_{0}, z_{1}, 1].
\]
Here each $z_{j}$ is a complex number which also represents a vector in $\mathbb{R}^{2}$, and $[a,b,c]$ represents the homogenous coordinates of $\mathbb{CP}^{2}$. To guarantee $\mathbb{CP}^{2} = \bigcup_{j=0}^{2} (M_{j} \setminus \partial M_{j})$, we have to let $r>1$. The larger $r$ is, the more overlap there will be.

Unlike the method in \cite{cao_qin}, we now need a partition of unity $\{ \rho_{j} \mid 0 \leq j \leq 2 \}$ satisfying Assumption \ref{asp_partition}. As mentioned before, it suffices to define the $\sigma_{j}$ in \eqref{eqn_rho}, which is in turn reduced to the definition of $\sigma_{j} \circ \phi_{j}$ on $D_{j}$. Let $r' = 0.9 r + 0.1$. Then $1< r' < r$. Suppose $x= (x_{1}, x_{2},x_{3}, x_{4}) \in D_{j} = [-r, r]^{4}$. Define
\begin{equation}\label{eqn_sigma}
\sigma_{j} \circ \phi_{j} (x) =
\begin{cases}
0, & \max \{ |x_{j}| \mid 1 \leq j \leq 4 \} > r'; \\
\prod_{j=1}^{4} (1- (\tfrac{x_{j}}{r'})^{2}), & \text{otherwise}.
\end{cases}
\end{equation}
Indeed, we can firstly define a function on $[-r, r]$ as
\[
w(a)=
\begin{cases}
0, & |a| > r'; \\
1- (\tfrac{a}{r'})^{2}, & \text{otherwise}.
\end{cases}
\]
Then $\sigma_{j} \circ \phi_{j} (x) = \prod_{j=1}^{4} w(x_{j})$. We see that $\mathrm{supp} \sigma_{j} = \phi_{j} ([-r', r']^{4}) \subset M_{j} \setminus \partial M_{j}$, $\sigma_{j}$ is Lipschitz continuous, and $\sum_{j=0}^{2} \sigma_{j} >0$. Hence the resulted family of $\rho_{j}$ satisfy Assumption \ref{asp_partition}.

For the discretization of $D_{j} = [-r, r]^{4}$, we divide each coordinate interval $[-r,r]$ into $N$ equal parts. The scale of the grid is thus $h=2r/N$. We keep $N \leq 80$ due to the memory limitation of the hardware.

To get the $n$-th numerical approximation $u_{h}^{n} = (u_{h,0}^{n}, u_{h,1}^{n}, u_{h,2}^{n})$, we need to solve a linear system $A_{j} X^{n}_{j} =b_{n,j}$ for $0 \leq j \leq 2$, where $X^{n}_{j}$ provides the interior degrees of freedom of $u_{h,j}^{n}$. We use the Conjugate Gradient Method (CG) to find $X^{n}_{j}$. As a result, the process to generate the sequence $\{ u_{h}^{n} \}$ is a nested iteration. The outer iteration is the DDM procedure Algorithm \ref{alg_discrete}. The initial guess is chosen as $u_{h}^{0} =0$. For each $n$, the inner iteration is the CG iteration to solve $A_{j} X^{n}_{j} =b_{n,j}$ for $0 \leq j \leq 2$. Note that $A_{j}$ remains the same when $n$ changes, whereas $b_{n,j}$ varies because of the evaluation of $u_{h,j}^{n}|_{\partial D_{j}}$. If $\{ u_{h}^{n} \}$ does converge, $X^{n-1}_{j}$ will be close to $X^{n}_{j}$ when $n$ is large enough. Thus, we choose the initial guess of $X^{n}_{j}$ as $X^{n-1}_{j}$. The tolerance for CG is set as
\[
\| A_{j} X^{n}_{j} - b_{n,j} \|_{2} / \| b_{n,j} \|_{2} \leq 10^{-8}.
\]
Our numerical results show that $u_{h}^{n}$ becomes stable when $n=n_{0}$ for some $n_{0}$, i.e. $u_{h}^{n} = u_{h}^{n_{0}}$ up to machine precision for all $n \geq n_{0}$. Actually, if
\[
\| A_{j} X^{n}_{j} - b_{n+1,j} \|_{2} / \| b_{n+1,j} \|_{2} \leq 10^{-8}
\]
for all $j$, then the inner iteration terminates for $n+1$ and $X^{n+1}_{j} = X^{n}_{j}$. We found that inner iteration terminates for all $n > n_{0}$. In other words, practically, the sequence $\{ u_{h}^{n} \}$ reaches its limit
\[
u_{h}^{\infty} = u_{h}^{n_{0}}
\]
at step $n_{0}$.

In the following tables,
\[
I_{h} u = (I_{h} u_{0}, I_{h} u_{1}, I_{h} u_{2}) \in V_{h},
\]
where $I_{h} u_{j} \in V_{h,j}$ is the interpolation of $u \circ \phi_{j}$ on $D_{j}$. We define the energy norm of the error as
\[
\| I_{h} u - u_{h}^{\infty} \|_{a} = \max \{ a_{j} (I_{h} u_{j} - u_{h,j}^{\infty}, I_{h} u_{j} - u_{h,j}^{\infty})^{\frac{1}{2}} \mid 0 \leq j \leq 2 \}.
\]
The $L^{2}$-norm $\| I_{h} u - u_{h}^{\infty} \|_{L^{2}}$, $L^{\infty}$-norm $\| I_{h} u - u_{h}^{\infty} \|_{L^{\infty}}$ and $H^{1}$-seminorm $| I_{h} u - u_{h}^{\infty} |_{H^{1}}$ are defined in similar ways. The numerical results are as follows in Tables \ref{tab_cp2_r1.2} and \ref{tab_cp2_r2}, where, for each norm, the data on the left side of each cell are errors and orders of convergence are appended to the right.

\begin{table}[htbp]
\begin{center}
\begin{tabular}{|c|cc|cc|cc|cc|c|} \hline
$h$ & \multicolumn{2}{c|}{$\|I_{h} u -  u_{h}^{\infty} \|_{L^{\infty}}$} & \multicolumn{2}{c|}{$\|I_{h} u -  u_{h}^{\infty} \|_{L^{2}}$} & \multicolumn{2}{c|}{$| I_{h} u - u_{h}^{\infty} |_{H^{1}}$} & \multicolumn{2}{c|}{$\| I_{h} u - u_{h}^{\infty} \|_{a}$} & $n_{0}$ \\
\hline
$0.24$ & $0.0376$ & & $0.0451$ & & $0.1548$ & & $0.0715$ & & $32$ \\
\hline
$0.12$ & $0.0103$ & $1.9$ & $0.0116$ & $2.0$ & $0.0438$ & $1.8$ & $0.0203$ & $1.8$ & $31$ \\
\hline
$0.06$ & $0.0024$ & $2.1$ & $0.0029$ & $2.0$ & $0.0126$ & $1.8$ & $0.0057$ & $1.8$ & $31$ \\
\hline
$0.03$ & $6.1841e-4$ & $2.0$ & $7.3551e-4$ & $2.0$ & $0.0039$ & $1.7$ & $0.0017$ & $1.7$ & $30$ \\
\hline
\end{tabular}
\end{center}
\caption{Numerical result on $\mathbb{CP}^{2}$ for $r=1.2$.}
\label{tab_cp2_r1.2}
\end{table}

\begin{table}[htbp]
\begin{center}
\begin{tabular}{|c|cc|cc|cc|cc|c|} \hline
$h$ & \multicolumn{2}{c|}{$\|I_{h} u -  u_{h}^{\infty} \|_{L^{\infty}}$} & \multicolumn{2}{c|}{$\|I_{h} u -  u_{h}^{\infty} \|_{L^{2}}$} & \multicolumn{2}{c|}{$| I_{h} u - u_{h}^{\infty} |_{H^{1}}$} & \multicolumn{2}{c|}{$\| I_{h} u - u_{h}^{\infty} \|_{a}$} & $n_{0}$ \\
\hline
$0.4$ & $0.1004$ & & $0.3600$ & & $0.7681$ & & $0.2134$ & & $10$ \\
\hline
$0.2$ & $0.0307$ & $1.7$ & $0.0826$ & $2.1$ & $0.2358$ & $1.7$ & $0.0647$ & $1.7$ & $11$ \\
\hline
$0.1$ & $0.0093$ & $1.7$ & $0.0219$ & $1.9$ & $0.0745$ & $1.7$ & $0.0192$ & $1.8$ & $11$ \\
\hline
$0.05$ & $0.0020$ & $2.2$ & $0.0052$ & $2.1$ & $0.0240$ & $1.6$ & $0.0059$ & $1.7$ & $11$ \\
\hline
\end{tabular}
\end{center}
\caption{Numerical result on $\mathbb{CP}^{2}$ for $r=2$.}
\label{tab_cp2_r2}
\end{table}

We see that the error $I_{h} u -  u_{h}^{\infty}$ decays in the optimal order when $h$ decreases.

\subsection{A Problem on \texorpdfstring{$B^{4}$}{B4}}
Let $M= B^{4}$ be the unit ball in $\mathbb{R}^{4}$, i.e.
\[
B^{4} = \left\{ (y_{1}, y_{2}, y_{3}, y_{4}) \in \mathbb{R}^{4} \middle| \sum_{i=1}^{4} y_{i}^{2} \leq 1 \right\}.
\]
This is a manifold with boundary $S^{3}$. Following the suggestion in Section \ref{sec_boundary}, to decompose $B^{4}$, we firstly choose a collar neighborhood of $\partial B^{4} = S^{3}$. Let
\begin{align*}
\psi:~ [\delta, 1] \times S^{3} & \rightarrow B^{4} \\
(t, y_{1}, y_{2}, y_{3}, y_{4}) & \mapsto t (y_{1}, y_{2}, y_{3}, y_{4})
\end{align*}
be a collar mapping, where $0 < \delta <1$. Let $M_{1} = [-s,s]^{4}$, we have
\[
B^{4} = \psi ((\delta,1] \times S^{3}) \cup (M_{1} \setminus \partial M_{1})
\]
provided that $0< \delta < s < 0.5$. To obtained a desired decomposition of the collar neighborhood, as pointed out in Section \ref{sec_boundary}, it suffices to decompose $S^{3}$. We employ the stereographic projection to do this job as in \cite[\S~5.1]{cao_qin}.

In summary, a decomposition with $3$ subdomains is provided by the following diffeomorphisms:
\begin{align*}
\phi_{1}: \ D_{1} = [-s,s]^{4} & \rightarrow M_{1} \\
x & \mapsto x,
\end{align*}
\begin{align*}
\phi_{2}: \ D_{2} = [\delta, 1] \times [-r,r]^{3} & \rightarrow M_{2} \\
(t,\check{x}) & \mapsto t \left( \frac{2\check{x}}{1 + \|\check{x}\|^{2}}, \frac{1 - \|\check{x}\|^{2}}{1 + \|\check{x}\|^{2}} \right),
\end{align*}
and
\begin{align*}
\phi_{3}: \ D_{3} = [\delta, 1] \times [-r,r]^{3} & \rightarrow M_{3} \\
(t,\check{x}) & \mapsto t \left( \frac{2\check{x}}{1 + \|\check{x}\|^{2}}, \frac{-1 + \|\check{x}\|^{2}}{1 + \|\check{x}\|^{2}} \right),
\end{align*}
where $\check{x}= (x_{1}, x_{2}, x_{3}) \in [-r,r]^{3}$ and $\| \check{x} \| = \sqrt{\sum_{i=1}^{3} x_{i}^{2}}$. The transitions of coordinates are as follows:
\[
\phi_{2}^{-1} \circ \phi_{1} (x) = \left( \|x\|, \frac{\check{x}}{\| x \| + x_{4}} \right) \quad \text{and} \quad \phi_{3}^{-1} \circ \phi_{1} (x) = \left( \|x\|, \frac{\check{x}}{\| x \| - x_{4}} \right),
\]
where $x= (x_{1}, x_{2}, x_{3}, x_{4})$ and $\check{x}= (x_{1}, x_{2}, x_{3})$ consists of the first three components of $x$;
\[
\phi_{1}^{-1} \circ \phi_{2} (t,\check{x}) = t \left( \frac{2\check{x}}{1 + \|\check{x}\|^{2}}, \frac{1 - \|\check{x}\|^{2}}{1 + \|\check{x}\|^{2}} \right),
\]
\[
\phi_{1}^{-1} \circ \phi_{3} (t,\check{x}) = t \left( \frac{2\check{x}}{1 + \|\check{x}\|^{2}}, \frac{-1 + \|\check{x}\|^{2}}{1 + \|\check{x}\|^{2}} \right),
\]
$\phi_{3}^{-1} \circ \phi_{2} = \phi_{2}^{-1} \circ \phi_{3}$, and
\[
\phi_{3}^{-1} \circ \phi_{2} (t,\check{x}) = \left( t, \frac{\check{x}}{\|\check{x}\|^{2}} \right).
\]
In the above, we have $0< \delta < s <0.5$ and $r>1$. The larger $s- \delta$ and $r$ are, the more overlap there will be.

To obtian the partition of unity $\{ \rho_{i} \mid 1 \leq i \leq 3 \}$, it suffices to define $\sigma_{j}$ in \eqref{eqn_sigma}. Let $\delta' = 0.9 \delta + 0.1 s$, $s' = 0.1 \delta + 0.9 s$, and $r' = 0.9 r + 0.1$. Define
\[
\sigma_{1} \circ \phi_{1} (x) =
\begin{cases}
0, & \max \{ |x_{j}| \mid 1 \leq j \leq 4 \} > s'; \\
\prod_{j=1}^{4} (1- (\tfrac{x_{j}}{s'})^{2}), & \text{otherwise};
\end{cases}
\]
$\sigma_{2} \circ \phi_{2} = \sigma_{3} \circ \phi_{3}$; and
\[
\sigma_{2} \circ \phi_{2} (t, \check{x}) =
\begin{cases}
0, & t< \delta'~\text{or}~\max \{ |x_{j}| \mid 1 \leq j \leq 3 \} > r'; \\
\frac{t-\delta'}{1- \delta'} \prod_{j=1}^{3} (1- (\tfrac{x_{j}}{r'})^{2}), & \text{otherwise}.
\end{cases}
\]

Equip $B^{4}$ with the Riemannian metric $g$ inherited from the standard one on $\mathbb{R}^{4}$. Then $\Delta$ on $B^{4}$ coincides with the usual $\Delta = \sum_{\alpha =1}^{4} \frac{\partial^{2}}{\partial y_{\alpha}^{2}}$ on $\mathbb{R}^{4}$. Clearly, on $D_{1}$, the metric is expressed as
\[
g= \sum_{\alpha =1}^{4} \mathrm{d} x_{\alpha} \otimes \mathrm{d} x_{\alpha}.
\]
The equation \eqref{eqn_problem_coordinate} on $D_{1}$ is: $\forall v \in H^{1}_{0} (D_{1})$,
\[
\int_{D_{1}} \left( \sum_{\alpha =1}^{4} \frac{\partial u \circ \phi_{1}}{\partial x_{\alpha}} \frac{\partial v}{\partial x_{\alpha}} + b \cdot u \circ \phi_{1} \cdot v \right) \mathrm{d} x_{1} \mathrm{d} x_{2} \mathrm{d} x_{3} \mathrm{d} x_{4} = \int_{D_{1}} f \circ \phi_{1} \cdot v \mathrm{d} x_{1} \mathrm{d} x_{2} \mathrm{d} x_{3} \mathrm{d} x_{4}.
\]
On $D_{2}$ and $D_{3}$, the Riemannian metric is expressed as
\[
g= \mathrm{d} t \otimes \mathrm{d} t + 4 t^{2} (1+ \| \check{x} \|^{2})^{-2} \sum_{\alpha =1}^{3} \mathrm{d} x_{\alpha} \otimes \mathrm{d} x_{\alpha}.
\]
The equation \eqref{eqn_problem_coordinate} on $D_{i}$ ($i=2,3$) is: $\forall v \in H^{1}_{0} (D_{i})$,
\begin{align*}
& \int_{D_{i}} \left[ 8 t^{3} (1+ \| \check{x} \|^{2})^{-3} \frac{\partial u \circ \phi_{i}}{\partial t} \frac{\partial v}{\partial t} +2t (1+ \| \check{x} \|^{2})^{-1} \sum_{\alpha =1}^{3} \frac{\partial u \circ \phi_{i}}{\partial x_{\alpha}} \frac{\partial v}{\partial x_{\alpha}} \right. \\
& \left. + 8 t^{3} (1+ \| \check{x} \|^{2})^{-3} b \cdot u \circ \phi_{i} \cdot v \right] \mathrm{d} t \mathrm{d} x_{1} \mathrm{d} x_{2} \mathrm{d} x_{3} \\
= &  \int_{D_{i}} 8 t^{3} (1+ \| \check{x} \|^{2})^{-3} f \circ \phi_{i} \cdot v \mathrm{d} t \mathrm{d} x_{1} \mathrm{d} x_{2} \mathrm{d} x_{3}.
\end{align*}

We choose the true solution $u$ in $B^{4}$ as
\[
u(y_{1}, y_{2}, y_{3}, y_{4}) = \sin (\pi y_{4}).
\]
Then $f$ in \eqref{eqn_problem} is
\[
f= (b+ \pi^{2}) u.
\]
In $D_{i}$, $u \circ \phi_{i}$ has the expression
\[
u \circ \phi_{1} (x) = \sin (\pi x_{4}),
\]
\[
u \circ \phi_{2} (t,x) = \sin \left( \pi t \frac{1 - \|\check{x}\|^{2}}{1 + \|\check{x}\|^{2}} \right),
\qquad
\text{and}
\qquad
u \circ \phi_{3} (t,x) = \sin \left( \pi t \frac{-1 + \|\check{x}\|^{2}}{1 + \|\check{x}\|^{2}} \right).
\]

Now we choose $b=0$ in \eqref{eqn_problem}. For the discretization on $D_{i}$, we divide each coordinate interval $[-s,s]$ and $[\delta, 1]$ into $N_{1}$ equal parts, divide each coordinate interval $[-r,r]$ into $N_{2}$ equal parts. Set $N_{1} = 0.4 N_{2}$. Thus the scale of the grid is
\[
h= \max \{ \tfrac{2s}{N_{1}}, \tfrac{1-\delta}{N_{1}}, \tfrac{2r}{N_{2}} \}.
\]
We choose $N_{2}$ as $10$, $20$, $40$, and $80$. For the initial guess, we set the degrees of freedom off $\partial M$ as $0$. The numerical results are as follows in Tables \ref{tab_b4_r1.2} and \ref{tab_b4_r2}.

\begin{table}[htbp]
\begin{center}
\begin{tabular}{|c|cc|cc|cc|cc|c|} \hline
$h$ & \multicolumn{2}{c|}{$\|I_{h} u -  u_{h}^{\infty} \|_{L^{\infty}}$} & \multicolumn{2}{c|}{$\|I_{h} u -  u_{h}^{\infty} \|_{L^{2}}$} & \multicolumn{2}{c|}{$| I_{h} u - u_{h}^{\infty} |_{H^{1}}$} & \multicolumn{2}{c|}{$\| I_{h} u - u_{h}^{\infty} \|_{a}$} & $n_{0}$ \\
\hline
$0.24$ & $0.1049$ & & $0.0604$ & & $0.3642$ & & $0.2278$ & & $13$ \\
\hline
$0.12$ & $0.0267$ & $2.0$ & $0.0177$ & $1.8$ & $0.1305$ & $1.5$ & $0.0799$ & $1.5$ & $13$ \\
\hline
$0.06$ & $0.0073$ & $1.9$ & $0.0047$ & $1.9$ & $0.0413$ & $1.7$ & $0.0251$ & $1.7$ & $13$ \\
\hline
$0.03$ & $0.0020$ & $1.9$ & $0.0012$ & $2.0$ & $0.0127$ & $1.7$ & $0.0078$ & $1.7$ & $13$ \\
\hline
\end{tabular}
\end{center}
\caption{Numerical result on $B^{4}$ for $(s, \delta, r) = (0.4, 0.2, 1.2)$.}
\label{tab_b4_r1.2}
\end{table}

\begin{table}[htbp]
\begin{center}
\begin{tabular}{|c|cc|cc|cc|cc|c|} \hline
$h$ & \multicolumn{2}{c|}{$\|I_{h} u -  u_{h}^{\infty} \|_{L^{\infty}}$} & \multicolumn{2}{c|}{$\|I_{h} u -  u_{h}^{\infty} \|_{L^{2}}$} & \multicolumn{2}{c|}{$| I_{h} u - u_{h}^{\infty} |_{H^{1}}$} & \multicolumn{2}{c|}{$\| I_{h} u - u_{h}^{\infty} \|_{a}$} & $n_{0}$ \\
\hline
$0.4$ & $0.2251$ & & $0.1389$ & & $0.6561$ & & $0.3443$ & & $8$ \\
\hline
$0.2$ & $0.0582$ & $2.0$ & $0.0418$ & $1.7$ & $0.2854$ & $1.2$ & $0.1182$ & $1.5$ & $9$ \\
\hline
$0.1$ & $0.0216$ & $1.4$ & $0.0120$ & $1.8$ & $0.1080$ & $1.4$ & $0.0361$ & $1.7$ & $9$ \\
\hline
$0.05$ & $0.0054$ & $2.0$ & $0.0031$ & $2.0$ & $0.0351$ & $1.6$ & $0.0111$ & $1.7$ & $9$ \\
\hline
\end{tabular}
\end{center}
\caption{Numerical result on $B^{4}$ for $(s, \delta, r) = (0.4, 0.1, 2)$.}
\label{tab_b4_r2}
\end{table}

\subsection{A Problem on \texorpdfstring{$B^{2} \times S^{2}$}{B2xS2}}
Let $M = B^{2} \times S^{2}$, where $B^{2}$ is the unit ball (disk) in $\mathbb{R}^{2}$ and $S^{2}$ is the unit sphere in $\mathbb{R}^{3}$. This is a manifold with boundary $S^{1} \times S^{2}$.

As pointed out in \cite[\S~4]{cao_qin}, decompositions of $B^{2}$ and $S^{2}$ yield a canonical product decomposition of $B^{2} \times S^{2}$. Thus it suffices to decompose both factor manifolds.

We use a decomposition of $B^{2}$ similar to that of $B^{4}$ above. In other words, it is provided by
\begin{align*}
\phi_{1}: \ D_{1} = [-s,s]^{2} & \rightarrow M_{1} \subset B^{2} \\
x & \mapsto x,
\end{align*}
\begin{align*}
\phi_{2}: \ D_{2} = [\delta, 1] \times [-r,r] & \rightarrow M_{2} \subset B^{2} \\
(t,\check{x}) & \mapsto t \left( \frac{2\check{x}}{1 + \check{x}^{2}}, \frac{1 - \check{x}^{2}}{1 + \check{x}^{2}} \right),
\end{align*}
and
\begin{align*}
\phi_{3}: \ D_{3} = [\delta, 1] \times [-r,r] & \rightarrow M_{3} \subset B^{2} \\
(t,\check{x}) & \mapsto t \left( \frac{2\check{x}}{1 + \check{x}^{2}}, \frac{-1 + \check{x}^{2}}{1 + \check{x}^{2}} \right).
\end{align*}
Here $0< \delta < s < \frac{1}{\sqrt{2}}$ and $r>1$. The transitions of coordinates are as follows:
\[
\phi_{2}^{-1} \circ \phi_{1} (x) = \left( \|x\|, \frac{x_{1}}{\| x \| + x_{2}} \right) \quad \text{and} \quad \phi_{3}^{-1} \circ \phi_{1} (x) = \left( \|x\|, \frac{x_{1}}{\| x \| - x_{2}} \right),
\]
where $x= (x_{1}, x_{2})$;
\[
\phi_{1}^{-1} \circ \phi_{2} (t,\check{x}) = t \left( \frac{2\check{x}}{1 + \check{x}^{2}}, \frac{1 - \check{x}^{2}}{1 + \check{x}^{2}} \right),
\qquad
\phi_{1}^{-1} \circ \phi_{3} (t,\check{x}) = t \left( \frac{2\check{x}}{1 + \check{x}^{2}}, \frac{-1 + \check{x}^{2}}{1 + \check{x}^{2}} \right),
\]
$\phi_{3}^{-1} \circ \phi_{2} = \phi_{2}^{-1} \circ \phi_{3}$, and
\[
\phi_{3}^{-1} \circ \phi_{2} (t,\check{x}) = \left( t, \frac{1}{\check{x}} \right).
\]

We use stereographic projection to decompose $S^{2}$. In other words, the decomposition is provided by
\begin{align*}
\phi'_{1}: \ D'_{1} = [-r,r]^{2} & \rightarrow M'_{1} \subset S^{2} \subset \mathbb{R}^{3}
\\
x' & \mapsto \left( \frac{2x'}{1 + \|x'\|^{2}}, \frac{1 - \|x'\|^{2}}{1 + \|x'\|^{2}} \right)
\end{align*}
and
\begin{align*}
\phi'_{2}: \ D'_{2} = [-r,r]^{2} & \rightarrow M'_{2} \subset S^{2} \subset \mathbb{R}^{3} \\
x' & \mapsto \left( \frac{2x'}{1 + \|x'\|^{2}}, \frac{-1 + \|x'\|^{2}}{1 + \|x'\|^{2}} \right).
\end{align*}
Here $r>1$. The transitions of coordinates are given by
\[
{\phi'}_{2}^{-1} \circ \phi'_{1} (x') = {\phi'}_{1}^{-1} \circ \phi'_{2} (x') = \frac{x'}{\| x' \|^{2}}.
\]

The above decompositions result in a decomposition of $B^{2} \times S^{2}$ with $3 \times 2 =6$ subdomains. Define $D_{i,i'} = D_{i} \times D'_{i'}$, $M_{i,i'} = M_{i} \times M'_{i'}$ and $\phi_{i,i'} = \phi_{i} \times \phi'_{i'}$. The decomposition of $B^{2} \times S^{2}$ are provided by the diffeomorphisms
\[
\phi_{i,i'}: D_{i,i'} \rightarrow M_{i,i'} \subset B^{2} \times S^{2}
\]
for $1 \leq i \leq 3$ and $1 \leq i' \leq 2$. The transition maps are
\[
\phi_{j,j'}^{-1} \circ \phi_{i,i'} = (\phi_{j} \times \phi'_{j'})^{-1} \circ (\phi_{i} \times \phi'_{i'}) = (\phi_{j}^{-1} \circ \phi_{i}) \times ({\phi'}_{j'}^{-1} \circ \phi'_{i'}).
\]
Note that $D_{1,1} = D_{1,2} = [-s,s]^{2} \times [-r,r]^{2}$ and other $D_{i,i'}$ are $[\delta, 1] \times [-r,r]^{3}$. Here $0< \delta < s < \frac{1}{\sqrt{2}}$ and $r>1$. The larger $s- \delta$ and $r$ are, the more overlap there will be.

We need to define a partition of unity $\{ \rho_{i,i'} \mid 1 \leq i \leq 3, 1 \leq i' \leq 2 \}$ such that $\mathrm{supp} \rho_{i,i'} \subset M_{i,i'} \setminus \overline{\gamma_{i,i'}}$, where $\gamma_{i,i'} = \partial M_{i,i'} \setminus \partial M$. Let $\delta' = 0.9 \delta + 0.1 s$, $s' = 0.1 \delta + 0.9 s$, and $r' = 0.9 r + 0.1$. On $D_{1,1} = D_{1,2} = [-s,s]^{2} \times [-r,r]^{2}$, define
\begin{align*}
& \sigma_{1,i'} \circ \phi_{1,i'} (x_{1}, x_{2}, x_{3}, x_{4}) \\
= &
\begin{cases}
0, & \max \{ |x_{1}|, |x_{2}| \} > s' ~\text{or}~ \max \{ |x_{3}|, |x_{4}| \} > r'; \\
\prod_{j=1}^{2} (1- (\tfrac{x_{j}}{s'})^{2}) \prod_{j=3}^{4} (1- (\tfrac{x_{j}}{r'})^{2}), & \text{otherwise}.
\end{cases}
\end{align*}
On other $D_{i,i'} = [\delta, 1] \times [-r,r]^{3}$, define
\begin{align*}
& \sigma_{i,i'} \circ \phi_{i,i'} (t, x_{1}, x_{2}, x_{3}) \\
= &
\begin{cases}
0, & t < \delta' ~\text{or}~ \max \{ |x_{1}|, |x_{2}|, |x_{3}| \} > r'; \\
\frac{t-\delta'}{1- \delta'} \prod_{j=1}^{3} (1- (\tfrac{x_{j}}{r'})^{2}), & \text{otherwise}.
\end{cases}
\end{align*}
Applying a formula similar to \eqref{eqn_rho}, we obtain a desired $\{ \rho_{i,i'} \mid 1 \leq i \leq 3, 1 \leq i' \leq 2 \}$.

Equip $B^{2}$ with the Riemannian metric inherited from the standard one on $\mathbb{R}^{2}$. Equip $S^{2}$ with the Riemannian metric inherited from the standard one on $\mathbb{R}^{3}$. Consider the product metric $g$ on $B^{2} \times S^{2}$. On $D_{1,i'}$, the coordinates are $(x,x')$, where $x= (x_{1}, x_{2})$ and $x'= (x'_{1}, x'_{2})$, the metric $g$ is expressed as
\[
g = \sum_{\alpha =1}^{2} \mathrm{d} x_{\alpha} \otimes \mathrm{d} x_{\alpha} + 4 (1+ \| x' \|^{2})^{-2} \sum_{\alpha' =1}^{2} \mathrm{d} x'_{\alpha'} \otimes \mathrm{d} x'_{\alpha'}.
\]
The equation \eqref{eqn_problem_coordinate} is
\begin{align*}
& \int_{D_{1,i'}} \left[ 4 (1+ \| x' \|^{2})^{-2} \sum_{\alpha =1}^{2} \frac{\partial u \circ \phi_{1,i'}}{\partial x_{\alpha}} \frac{\partial v}{\partial x_{\alpha}} + \sum_{\alpha' =1}^{2} \frac{\partial u \circ \phi_{1,i'}}{\partial x'_{\alpha'}} \frac{\partial v}{\partial x'_{\alpha'}} \right. \\
& \left. + 4 (1+ \| x' \|^{2})^{-2} b \cdot u \circ \phi_{1,i'} \cdot v \right] \mathrm{d} x_{1} \mathrm{d} x_{2} \mathrm{d} x'_{1} \mathrm{d} x'_{2} \\
= & \int_{D_{1,i'}} 4 (1+ \| x' \|^{2})^{-2} \cdot f \circ \phi_{1,i'} \cdot v \mathrm{d} x_{1} \mathrm{d} x_{2} \mathrm{d} x'_{1} \mathrm{d} x'_{2}.
\end{align*}
On $D_{i,i'}$ for $2 \leq i \leq 3$, the coordinates are $(t, \check{x}, x')$, where $x'= (x'_{1}, x'_{2})$,
\[
g = \mathrm{d} t \otimes \mathrm{d} t + 4 t^{2} (1+ \check{x}^{2})^{-2} \mathrm{d} \check{x} \otimes \mathrm{d} \check{x} + 4 (1+ \| x' \|^{2})^{-2} \sum_{\alpha' =1}^{2} \mathrm{d} x'_{\alpha'} \otimes \mathrm{d} x'_{\alpha'}.
\]
The equation \eqref{eqn_problem_coordinate} is
\begin{align*}
& \int_{D_{i,i'}} \left[ 8t (1+ \check{x}^{2})^{-1} (1+ \| x' \|^{2})^{-2} \frac{\partial u \circ \phi_{i,i'}}{\partial t} \frac{\partial v}{\partial t} \right. \\
& + 2t^{-1} (1+ \check{x}^{2}) (1+ \| x' \|^{2})^{-2} \frac{\partial u \circ \phi_{i,i'}}{\partial \check{x}} \frac{\partial v}{\partial \check{x}} + 2t (1+ \check{x}^{2})^{-1} \sum_{\alpha' =1}^{2} \frac{\partial u \circ \phi_{i,i'}}{\partial x'_{\alpha'}} \frac{\partial v}{\partial x'_{\alpha'}} \\
& \left. + 8t (1+ \check{x}^{2})^{-1} (1+ \| x' \|^{2})^{-2} b \cdot u \circ \phi_{i,i'} \cdot v \right] \mathrm{d} t \mathrm{d} \check{x} \mathrm{d} x'_{1} \mathrm{d} x'_{2} \\
= & \int_{D_{i,i'}} 8t (1+ \check{x}^{2})^{-1} (1+ \| x' \|^{2})^{-2} \cdot f \circ \phi_{i,i'} \cdot v \mathrm{d} t \mathrm{d} \check{x} \mathrm{d} x'_{1} \mathrm{d} x'_{2}.
\end{align*}

We set the true solution $u$ on $B^{2} \times S^{2}$ as
\[
u(y_{1}, y_{2}, y'_{1}, y'_{2}, y'_{3}) = \sin (\pi y_{2}) + y'_{3}.
\]
Then $f$ in \eqref{eqn_problem} is
\[
f(y_{1}, y_{2}, y'_{1}, y'_{2}, y'_{3}) = (b+ \pi^{2}) \sin (\pi y_{2}) + (b+2) y'_{3}.
\]
On $D_{i,i'}$, we have
\[
u \circ \phi_{i,i'} (x_{1}, x_{2}, x_{3}, x_{4}) = u_{i} (x_{1}, x_{2}) + u'_{i'} (x_{3}, x_{4})
\]
and
\[
f \circ \phi_{i,i'} (x_{1}, x_{2}, x_{3}, x_{4}) = f_{i} (x_{1}, x_{2}) + f'_{i'} (x_{3}, x_{4}),
\]
where
\[
u_{1} (x_{1}, x_{2}) = \sin (\pi x_{2}), \quad u_{2} (t, \check{x}) = \sin \left( \pi t \frac{1- \check{x}^{2}}{1+ \check{x}^{2}} \right), \quad u_{3} (t, \check{x}) = \sin \left( \pi t \frac{-1+ \check{x}^{2}}{1+ \check{x}^{2}} \right),
\]
\[
u'_{1} (x') = \frac{1- \| x' \|^{2}}{1+ \| x' \|^{2}}, \qquad u'_{2} (x') = \frac{-1 + \| x' \|^{2}}{1+ \| x' \|^{2}},
\]
$f_{i} = (b+ \pi^{2}) u_{i}$, and $f'_{i'} = (b+ 2) u'_{i'}$.

Now we choose $b=1$ in \eqref{eqn_problem}. For the discretization on $D_{i,i'}$, we divide each coordinate interval $[-s,s]$ and $[\delta, 1]$ into $N_{1}$ equal parts, divide each coordinate interval $[-r,r]$ into $N_{2}$ equal parts. Set $N_{1} = 0.4 N_{2}$. We choose $N_{2}$ as $10$, $20$, $40$, and $80$. For the initial guess, we set the degrees of freedom off $\partial M$ as $0$. The numerical results are as follows in Tables \ref{tab_b2s2_r1.2} and \ref{tab_b2s2_r2}.

\begin{table}[htbp]
\begin{center}
\begin{tabular}{|c|cc|cc|cc|cc|c|} \hline
$h$ & \multicolumn{2}{c|}{$\|I_{h} u -  u_{h}^{\infty} \|_{L^{\infty}}$} & \multicolumn{2}{c|}{$\|I_{h} u -  u_{h}^{\infty} \|_{L^{2}}$} & \multicolumn{2}{c|}{$| I_{h} u - u_{h}^{\infty} |_{H^{1}}$} & \multicolumn{2}{c|}{$\| I_{h} u - u_{h}^{\infty} \|_{a}$} & $n_{0}$ \\
\hline
$0.3$ & $0.1046$ & & $0.0916$ & & $0.3938$ & & $0.4863$ & & $35$ \\
\hline
$0.15$ & $0.0343$ & $1.6$ & $0.0261$ & $1.8$ & $0.1180$ & $1.7$ & $0.1373$ & $1.8$ & $36$ \\
\hline
$0.075$ & $0.0083$ & $2.0$ & $0.0072$ & $1.9$ & $0.0332$ & $1.8$ & $0.0373$ & $1.9$ & $33$ \\
\hline
$0.0375$ & $0.0021$ & $2.0$ & $0.0018$ & $2.0$ & $0.0103$ & $1.7$ & $0.0113$ & $1.7$ & $31$ \\
\hline
\end{tabular}
\end{center}
\caption{Numerical result on $B^{2} \times S^{2}$ for $(s, \delta, r)= (0.6,0.3,1.2)$.}
\label{tab_b2s2_r1.2}
\end{table}

\begin{table}[htbp]
\begin{center}
\begin{tabular}{|c|cc|cc|cc|cc|c|} \hline
$h$ & \multicolumn{2}{c|}{$\|I_{h} u -  u_{h}^{\infty} \|_{L^{\infty}}$} & \multicolumn{2}{c|}{$\|I_{h} u -  u_{h}^{\infty} \|_{L^{2}}$} & \multicolumn{2}{c|}{$| I_{h} u - u_{h}^{\infty} |_{H^{1}}$} & \multicolumn{2}{c|}{$\| I_{h} u - u_{h}^{\infty} \|_{a}$} & $n_{0}$ \\
\hline
$0.4$ & $0.2555$ & & $0.4455$ & & $1.7415$ & & $1.3294$ & & $18$ \\
\hline
$0.2$ & $0.0831$ & $1.6$ & $0.1344$ & $1.7$ & $0.6474$ & $1.4$ & $0.4392$ & $1.6$ & $18$ \\
\hline
$0.1$ & $0.0229$ & $1.9$ & $0.0247$ & $2.4$ & $0.1639$ & $2.0$ & $0.1320$ & $1.7$ & $19$ \\
\hline
$0.05$ & $0.0059$ & $2.0$ & $0.0061$ & $2.0$ & $0.0519$ & $1.7$ & $0.0415$ & $1.7$ & $19$ \\
\hline
\end{tabular}
\end{center}
\caption{Numerical result on $B^{2} \times S^{2}$ for $(s, \delta, r)= (0.7,0.1,2)$.}
\label{tab_b2s2_r2}
\end{table}

\section*{Acknowledgements}
We thank Jun Li, Yuanming Xiao and Yiyan Xu for various discussions. We thank the anonymous reviewers for many constructive comments and suggestions which led to an improved presentation of this paper. The first author was partially supported by NSFC 11871272. The second author was partially supported by NSFC 12071227 and 12371369, and the National Key Research and Development Program of China 2020YFA0713803. The third author was partially supported by NSFC 12171349 and 12271389.

\end{document}